\documentclass{article}
\usepackage{graphicx} % Required for inserting images
\usepackage{fullpage}
\usepackage{amssymb}
\usepackage{amsmath}
\usepackage{amsthm}
\usepackage{stmaryrd}
\usepackage{xypic}
\usepackage{setspace}
\usepackage{array}
\usepackage{rotating}
\usepackage{verbatim}
\usepackage{tikz}
\usepackage[colorlinks=true,urlcolor=blue]{hyperref}
\usepackage[capitalise]{cleveref}
\usepackage[inline,shortlabels]{enumitem}   
\usepackage{authblk}

\usepackage[ shadow, colorinlistoftodos,textsize=tiny, obeyFinal]{todonotes}
\setuptodonotes{fancyline, backgroundcolor=gray!10,bordercolor=gray}

\newtheorem{theorem}{Theorem}[section]
\newtheorem{lemma}[theorem]{Lemma}
\newtheorem{proposition}[theorem]{Proposition}
\newtheorem{corollary}[theorem]{Corollary}
\newtheorem{definition}[theorem]{Definition}
\newtheorem{example}[theorem]{Example}

\newtheorem{problem}{Problem}

\newcommand\xxoverset[3]{%
  \resizebox{#1+\widthof{\scriptsize #2}}{\height}{$#3$}}
\newcommand\extoverset[3][0pt]{%
  \mathrel{\overset{\textup{#2}}{\xxoverset{#1}{#2}{#3}}}}

\newcommand{\calK}{\mathcal K}
\newcommand{\calL}{\mathcal L}
\newcommand{\calM}{\mathcal M}
\newcommand{\calN}{\mathcal N}
\newcommand{\calP}{\mathcal P}
\newcommand{\calQ}{\mathcal Q}

\newcommand{\calR}{\mathcal R}
\newcommand{\calT}{\mathcal T}
\newcommand{\mix}{\mathbin{\diamondsuit}}

\newcommand{\ijpathSet}[1]{\extoverset{~#1~}{-}}
\newcommand{\ijpath}[2]{\ijpathSet{[#1, #2]}}

\newcommand{\UVG}[2]{X_{\wedge}(#1 | #2)}
\newcommand{\UVGM}{\UVG{\calM_1}{\calM_2}}
\newcommand{\LVG}[2]{X_{\vee}(#1 | #2)}
\newcommand{\LVGM}{\LVG{\calM_1}{\calM_2}}

\newcommand{\VG}[2]{X(#1 | #2)}

\newcommand{\VGK}{\VG{\calK_1}{\calK_2}}
\newcommand{\VGL}{\VG{\calL_1}{\calL_2}}
\newcommand{\VGN}{\VG{\calN_1}{\calN_2}}

\newcommand{\VGKb}{\VG{\calK_2}{\calK_1}}
\newcommand{\VGLb}{\VG{\calL_2}{\calL_1}}
\newcommand{\VGNb}{\VG{\calN_2}{\calN_1}}

\newcommand{\G}{\Gamma}

\DeclareMathOperator{\Conn}{Conn}
\DeclareMathOperator{\Stab}{Stab}
\DeclareMathOperator{\Norm}{Norm}
\DeclareMathOperator{\lcm}{lcm}
\DeclareMathOperator{\rank}{rank}

\title{Polytopality criteria for the mix of polytopes and maniplexes}
\author[1]{Gabe Cunningham}
\author[2]{Isabel Hubard}
\affil[1]{School of Computing and Data Science, Wentworth Institute of Technology, Boston, MA 02115}
\affil[2]{ Institute of Mathematics, National Autonomous University of Mexico (IM UNAM), 04510 Mexico City, Mexico}
\date{\today}

\begin{document}

\maketitle

\begin{abstract}
The \emph{mix} of two maniplexes is the minimal maniplex that covers both. This construction has many important applications, such as finding the smallest regular cover of a maniplex. If one of the maniplexes is an abstract polytope, a natural question to ask is whether the mix is also a polytope. We describe here a general criterion for the polytopality of the mix which generalizes several previously-known polytopality criteria.
\end{abstract}

%\listoftodos

\section{Introduction}

Abstract polytopes are an expansion of the combinatorial type of a convex polytope, allowing a wide variety of objects while still retaining a geometric flavor. An abstract polytope can be represented as a partially-ordered set of its faces, or as an edge-colored graph that describes how flags (maximal chains in the poset) are related. In the case of highly-symmetric polytopes, they can also be described by a presentation of their automorphism group with respect to a particular generating set. This flexibility of representation has helped the field grow in many different directions.

The representation as an edge-colored graph has gained prominence recently, and \emph{maniplexes} were introduced as a generalization of abstract polytopes \cite{maniplexes}. This has proven especially useful since there are many natural operations on polytopes that might not produce a polytope, but they almost always produce a maniplex. Conversely, most operations on polytopes can be naturally extended to maniplexes, and so the study of maniplexes is growing, even independent of polytopes.

A maniplex of rank $n$ is the flag graph of a polytope if and only if it satisfies the \emph{Path Intersection Property} \cite{poly-mani}, that whenever there is a path between flags $\Phi$ and $\Psi$ that uses colors in $\{0, 1, \ldots, j\}$ and a path that uses colors in $\{i, i+1, \ldots, n-1\}$, then there is a path that uses colors in $\{i, \ldots, j\}$. In practice though, we often build polytopes inductively, by assembling rank $n-1$ polytopes together into a rank $n$ polytope. Thus, it is useful to have a version of the Path Intersection Property that assumes that the pieces we are gluing together are themselves polytopes. We prove two such recursive versions of the Path Intersection Property in \cref{recusive-pip-facet,recursive-pip}. 

One operation on maniplexes that is particularly useful is their \emph{mix}, a natural generalization of the \emph{join} or \emph{parallel product} of maps on surfaces (see \cite{parallel-product}). The mix of two maniplexes is the smallest maniplex that naturally covers both -- in other words, it is the fiber product in the category of (rooted) maniplexes of a fixed rank. Mixing is useful for a variety of purposes, such as constructing the smallest regular (fully symmetric) cover of a given maniplex. See \cite{chiral-mix,var-gps,mixing-and-monodromy} for a variety of applications of the mix.

Naturally, if we view polytopes as maniplexes, we can mix them as well, but a complication arises: the mix of polytopes may not be a polytope. We thus seek simple criteria for when the mix of two polytopes -- or a polytope with a maniplex -- will be a polytope (see \cite[Problem 9]{k-orbit}). Several such results have appeared in the literature before, focused on the most symmetric polytopes: regular and chiral (for example, see \cite[Lemma 3.3]{chiral-mix} and \cite[Theorem 3.1]{chiral-covers}).
%Here, we describe a single framework using \emph{variance groups} (see \cite{var-gps}) that unifies and extends many of these results. 
Building on the concept of variance groups (see \cite{var-gps}), introduced to measure how "different" two structures are with respect to their mix, we present a unified and general approach to determining when the mix of two maniplexes is itself a polytope.
Our framework unifies and extends many of the known results. 

The main results of this paper are \cref{thm:all-i-admissible-mix,cor:med-sec-trans-I,cor:2_0n-1,cor:2_0,cor:2_n-1}, which in particular apply to the mix of regular polytopes, or two-orbit polytopes in the same class. 
We examine several further consequences of \cref{thm:all-i-admissible-mix}, including a generalization of \cite{chiral-covers} in \cref{cor:chiral-like-src}. 

In addition to variance groups, one of the important tools we use is the notion of a maniplex $\calM$ being \emph{$\calT$-admissible}, where $\calT$ is a premaniplex (a quotient of a maniplex). 
This idea was introduced for maps in \cite{map-operations}, but has been relatively underused. 
In fact, we now have a clear way to think about it: to say that $\calM$ is $\calT$-admissible is to say that $\calT$ is a \emph{symmetry type graph} of $\calM$ \cite{stg}. 
In other words, there is a subgroup $\G$ of the automorphism group of $\calM$ such that $\calM / \G \cong \calT$. This way of thinking lets us bring the considerable toolbox of symmetry type graphs to bear, and it seems clear that $\calT$-admissibility will continue to be an important concept in the further study of maniplexes and polytopes.

The paper is organized as follows. 
In \cref{sec:background} we give the basic notions of polytopes and maniplexes and some connections between them, as well as symmetry notions. 
In particular, in \cref{sec:maniplexes} we give two recursive versions of the Path Intersection Property.
\cref{sec:mixing} describes the mix of two maniplexes and a version of the Path Intersection Property for the mix. \cref{sec:t-admissibility} provides further information on the mix and introduces the notion of $\calT$-admissibility for maniplexes. In \cref{sec:var-gps}, we define chirality groups and variance groups, and in \cref{sec:mix-i-admissible}, we prove our main result on the polytopality of the mix and explore several consequences. We conclude with some open problems for further study.

\section{Polytopes and maniplexes}
\label{sec:background}

\subsection{Polytopes}
Let us start by defining the basics of abstract polytopes, following \cite{arp}. A \emph{flagged poset of rank $n$} is a ranked poset $\calP$ with a unique minimal element of rank $-1$, a unique maximal element of rank $n$, and such that every maximal chain has exactly $n+2$ elements, one in each rank $\{-1, 0, \ldots, n\}$. A flagged poset is \emph{connected} if it is rank $1$ or less, or if its Hasse diagram is connected even after removal of the minimal and maximal element. If $F,G \in \calP$ with $F \leq G$, then the \emph{section} $G/F$ is the interval $[F,G] = \{H \mid F \leq H \leq G\}$. 

An \emph{(abstract) $n$-polytope} is a flagged poset $\calP$ of rank $n$ such that:
\begin{enumerate}
    \item Every section of $\calP$ is connected, and
    \item If $F \leq G$ with $\rank{G} = \rank{F} + 2$, then there are exactly two elements $H$ such that $F < H < G$.
\end{enumerate}

The second property is often referred to as the {\it diamond condition}.
The elements of an abstract polytope are called \emph{faces}, and a face of rank $i$ is also called an \emph{$i$-face}. Faces of rank $0$, $1$, and $n-1$ are called vertices, edges, and facets, respectively. Faces $F$ and $G$ such that $F \leq G$ are said to be \emph{incident}.

If $I \subseteq \{0, \ldots, n-1\}$, then a \emph{chain of type $I$} is a set of mutually incident faces, with one face of rank $i$ for each $i \in I$. The maximal chains of a polytope (with $I = \{0, \ldots, n-1\}$) are called \emph{flags}. 
By the diamond condition, for every $0 \leq i \leq n-1$ and every flag $\Phi$, there is a unique flag $\Phi^i$ that differs from $\Phi$ only in its $i$-face. We say that $\Phi^i$ is \emph{$i$-adjacent} to $\Phi$. More generally, $\Phi$ and $\Psi$ are \emph{adjacent} if they are $i$-adjacent for some $i$. We extend the exponential notation inductively by defining $\Phi^{i_1, \ldots, i_k}$ to be $(\Phi^{i_1, \ldots, i_{k-1}})^{i_k}$. Note that if $|i-j| > 1$, then $\Phi^{i,j} = \Phi^{j,i}$.

The sections of an abstract polytope are themselves abstract polytopes. Often, when we talk about the facets of an abstract polytope, we really have in mind the sections $F/F_{-1}$, where $F$ is a facet and $F_{-1}$ is the unique minimal face. The \emph{vertex-figure at $v$} is the section $F_n / v$, where $F_n$ is the maximal face. Note that if $\calP$ is an $n$-polytope, then its facets and vertex-figures are $(n-1)$-polytopes. A section of the form $F/v$ where $F$ is a facet and $v$ a vertex with $v \leq F$ is called a \emph{medial section} of $\calP$. Note that the medial sections of $\calP$ are precisely the facets of the vertex-figures, as well as the vertex-figures of the facets.

The \emph{flag graph} of $\calP$ is a graph whose vertices are the flags of $\calP$ and where flags that are $i$-adjacent are connected by an edge of color $i$. This is an $n$-regular graph, where each flag is incident to one edge of each color in $\{0, \ldots, n-1\}$. Furthermore, whenever $|i-j| > 1$, the edges of color $i$ and $j$ comprise a union of 4-cycles.

We will describe in \cref{sec:maniplexes} how to recover the poset from the flag graph. Thus, we will identify an abstract polytope with its flag graph, thinking about posets only when that is helpful.

\subsection{Maniplexes}
\label{sec:maniplexes}

Maniplexes were defined in \cite{maniplexes} as a generalization of (the flag graph of) polytopes. Formally, an \emph{$n$-maniplex} is a simple, connected, $n$-regular graph, with edges colored $\{0, \ldots, n-1\}$, such that every vertex is incident to an edge of each color, and such that the graph induced by the edges of colors $i$ and $j$ is a collection of 4-cycles if $|i-j| > 1$. The vertices of a maniplex are called \emph{flags}. Thus, the flag graph of an $n$-polytope is an $n$-maniplex.

We sometimes want to deal with quotients of a maniplex, and the resulting graph may no longer be simple. 
In particular, we may obtain \emph{multiple edges} (distinct edges with the same vertex sets), or \emph{semi-edges} (edges on a single vertex). Furthermore, the connected components that we get from keeping only the edges of colors $i$ and $j$ with $|i-j| > 1$ now only need to be quotients of $4$-cycles. 
We call these structures \emph{premaniplexes}. 
In this paper we will assume that all our premaniplexes are connected, though it is sometimes convenient to allow disconnected premaniplexes (as in \cite{voltage-ops}). 

The \emph{dual} of an $n$-premaniplex $\calM$ is the premaniplex $\calM^\delta$ we obtain by changing every edge color $i$ to $n-i-1$. Thus, the facets of $\calM^\delta$ correspond to the vertex-figures of $\calM$. 
Whenever we say that something follows from \emph{a dual argument}, we mean that applying the argument to $\calM^\delta$ gives the required result for $\calM$.

Whenever there is a path in a premaniplex $\calM$ between flags $\Phi$ and $\Psi$ that only uses colors in some set $S$, we will write $\Phi \ijpathSet{S} \Psi$. We will write $[i,j]$ for the set $\{i, i+1, \ldots, j\}$. 

We define a poset from an $n$-maniplex $\calM$ as follows. For $0 \leq i \leq n-1$, the set of $i$-faces consists of the connected components that remain after removing the edges of color $i$. If $F$ is an $i$-face and $G$ is a $j$-face, then $F \leq G$ if and only if $i \leq j$ and $F \cap G \neq \emptyset$. To finish the poset, we add a unique minimal element of rank $-1$ and maximal element of rank $n$. If we apply this operation to the flag graph of a polytope $\calP$, then we get $\calP$ back. More generally, we would like to know when a maniplex is the flag graph of a polytope. We say that a maniplex is \emph{polytopal} if it is isomorphic to the flag graph of a polytope. The following result characterizes which maniplexes are polytopal.

\begin{theorem}[{{\cite[Theorem 5.3]{poly-mani}}}]
\label{pip}
The maniplex $\calM$ is polytopal if and only if it satisfies the \emph{Path Intersection Property}: that whenever $\Phi \ijpath{0}{j} \Psi$ and $\Phi \ijpath{i}{n-1} \Psi$, then $\Phi \ijpath{i}{j} \Psi$, for every $0 \leq i < j \leq n-1$.
\end{theorem}

In an $n$-premaniplex $\calM$, the edges of color $i$ induce a permutation of the flags $s_i$. 
The group $\langle s_0, \ldots, s_{n-1} \rangle$ is called the \emph{connection group} or \emph{monodromy group} of $\calM$, which we will denote $\Conn(\calM)$. 
The elements of $\Conn(\calM)$ will be called \emph{connections}. 
Note that $s_i^2 = 1$ for each $i$, and $(s_i s_j)^2 = 1$ whenever $|i-j| > 1$. Furthermore, $\calM$ is a maniplex if and only if each $s_i$ and $s_i s_j$ is fixed-point-free. 

We may take a slightly different viewpoint by considering the (infinite) Coxeter group
\[ W_n = \langle r_0, \ldots, r_{n-1} \mid r_i^2 = 1 \text{ for } 0 \leq i \leq n-1, (r_i r_j)^2 = 1 \text{ for } 0 \leq i < j-1 \leq n-2 \rangle. \]
Then there is a well-defined action of $W_n$ on the flags of an $n$-premaniplex $\calM$, given by $r_i \Phi = \Phi^i$. 
Thus, there is a group epimorphism from $W_n$ to $\Conn(\calM)$ sending $r_i$ to $s_i$. Note that $\Phi \ijpath{i}{j} \Psi$ if and only if $\Psi \in \langle r_i, \ldots, r_j \rangle \Phi = \langle s_i, \ldots, s_j \rangle \Phi$.

    Let us refine the Path Intersection Property (\cref{pip}) and consider the case where we know that the facets and/or vertex-figures are polytopal.

    \begin{proposition}
    \label{recusive-pip-facet}
    Let $\calM$ be an $n$-maniplex. Then $\calM$ is polytopal if and only if the facets are polytopal and, for every pair of flags $\Lambda$ and $\Psi$ and every $0 \leq i \leq n-2$, if $\Lambda \ijpath{0}{n-2} \Psi$ and $\Lambda \ijpath{i}{n-1} \Psi$, then $\Lambda \ijpath{i}{n-2} \Psi$. 
    \end{proposition}

    \begin{proof}
    If $\calM$ is polytopal, then clearly its facets are polytopal and it satisfies the other property by \cref{pip}.

    Conversely, suppose that $\calM$ satisfies the properties above. We want to show that $\calM$ satisfies the Path Intersection Property. Consider flags $\Lambda$ and $\Psi$ and $1 \leq i < j \leq n-1$ such that $\Lambda \ijpath{0}{j} \Psi$ and $\Lambda \ijpath{i}{n-1} \Psi$. We need to show that $\Lambda \ijpath{i}{j} \Psi$. If $j = n-1$, there is nothing to show. Otherwise, if $j \leq n-2$, then in particular $\Lambda \ijpath{0}{n-2} \Psi$, and so by assumption, $\Lambda \ijpath{i}{n-2} \Psi$. Thus $\Lambda$ and $\Psi$ belong to the same facet of $\calM$, and since $\calM$ is polytopal, the facts that $\Lambda \ijpath{i}{n-2} \Psi$ and $\Lambda \ijpath{0}{j} \Psi$ imply that $\Lambda \ijpath{i}{j} \Psi$.
    \end{proof}
    
    \begin{proposition} \label{recursive-pip}
    Let $\calM$ be an $n$-maniplex. Then $\calM$ is polytopal if and only if
    \begin{enumerate}
        \item The facets are polytopal,
        \item The vertex-figures are polytopal, and
        \item For every pair of flags $\Lambda$ and $\Psi$, if $\Lambda \ijpath{0}{n-2} \Psi$ and $\Lambda \ijpath{1}{n-1} \Psi$, then $\Lambda \ijpath{1}{n-2} \Psi$. 
    \end{enumerate}
    \end{proposition}
    
    \begin{proof}
    If $\calM$ is polytopal, then clearly it satisfies the first two properties, and it satisfies the third by \cref{pip}.
    
    Conversely, suppose that $\calM$ satisfies the properties above. It suffices to show that $\calM$ satisfies the conditions of \cref{recusive-pip-facet}. So, consider flags $\Lambda$ and $\Psi$ and $0 \leq i \leq n-2$ such that $\Lambda \ijpath{0}{n-2} \Psi$ and $\Lambda \ijpath{i}{n-1} \Psi$. We want to show that $\Lambda \ijpath{i}{n-2} \Psi$. If $i = 0$, there is nothing to show. Otherwise, if $i \geq 1$, then in particular $\Lambda \ijpath{1}{n-1} \Psi$. Then by assumption, $\Lambda \ijpath{1}{n-2} \Psi$. Since the vertex-figures are polytopes and thus satisfy the full Path Intersection Property, the facts that $\Lambda \ijpath{i}{n-1} \Psi$ and $\Lambda \ijpath{1}{n-2} \Psi$ imply that $\Lambda \ijpath{i}{n-2} \Psi$.
    \end{proof}

\subsection{Maniplex automorphisms and regularity}
\label{sec:reg-mans}

A function $\varphi: \calM \to \calM$ is a \emph{(premaniplex) automorphism} if is a bijection of the flags such that $\Phi = \Psi^i$ if and only if $\Phi \varphi = (\Psi \varphi)^i$. That is, it is a graph automorphism that also preserves the colors of the edges. In particular, the action of the automorphism group commutes with the action of the connection group (and $W_n$). We write automorphisms as acting on the right, while connections act on the left. The automorphism group of a premaniplex $\calM$ is denoted $\G(\calM)$. 

If $\G \leq \G(\calM)$, then the \emph{symmetry type graph of $\calM$ with respect to $\G$} is the quotient $\calM/\G$. That is, it is the graph whose vertices are $\G$-orbits of flags, and with two orbits connected with an edge labeled $i$ if there is a flag in the first orbit that is $i$-adjacent to a flag in the second orbit. In fact, since automorphisms commute with connections, whenever this occurs then every flag in the first orbit is $i$-adjacent to every flag in the second orbit. A symmetry type graph of a premaniplex is itself a premaniplex. When we take $\G = \G(\calM)$, then the resulting quotient is referred to as \emph{the} symmetry type graph of $\calM$. 
One of the many properties of $\calM$ that can be obtained from its symmetry type graphs is that of transitivity (see \cite{stg} for details).
$\calM$ is transitive on $i$-faces if and only if there is a symmetry type graph  of $\calM$ such that it remains connected when deleting all (semi)edges of color $i$.
Further, $\calM$ is transitive on chains of type $I$ if and only if there is a symmetry type graph  of $\calM$ such that it remains connected when deleting all (semi)edges of colors $i$, for all $i\in I$.

Many other properties of $\calM$ can be obtained from its symmetry type graph; see \cite{stg} for more details.

By connectivity, the automorphism group of a premaniplex acts semiregularly on the flags. Thus, if we fix a \emph{base flag} $\Phi$, we can identify an automorphism by where it sends $\Phi$. Different choices of $\Phi$ can lead to different natural generating sets and different presentations for the automorphism group, so when we wish to emphasize this connection, we write $\G(\calM, \Phi)$ to denote the automorphism group of $\calM$ with automorphisms named according to their action on $\Phi$. We may also write $(\calM, \Phi)$ to refer to $\calM$ if we want to emphasize the base flag, but we often prefer the streamlined notation of just using $\calM$.

We say that $\calM$ is a \emph{$k$-orbit premaniplex} when $\G(\calM)$ has $k$ orbits on the flags. Whenever $\G(\calM)$ acts transitively (and thus regularly) on the flags of $\calM$, the premaniplex is called \emph{regular} or \emph{reflexible}. If $\calM$ is a regular $n$-premaniplex with base flag $\Phi$, then for each $0 \leq i \leq n-1$, there is an automorphism $\rho_i$ such that $\Phi \rho_i = \Phi^i$. These automorphisms generate the group so that
\[ \G(\calM) = \langle \rho_0, \ldots, \rho_{n-1} \rangle. \]
Similar to $\Conn(\calM)$ and $W_n$, these generators satisfy $\rho_i^2 = id$ for each $i$, and $(\rho_i \rho_j)^2 = id$ if $|i-j| > 1$. In fact, the map $r_i \mapsto \rho_i$ defines an isomorphism between $\Conn(\calM)$ and $\G(\calM)$ for regular premaniplexes. Note that, for regular premaniplexes, the choice of base flag is essentially arbitrary; different choices will lead to groups with identical presentations.

A group $\langle x_0, \ldots, x_{n-1} \rangle$ that satisfies $x_i^2 = 1$ and $(x_i x_j)^2 = 1$ whenever $|i-j| > 1$ is called a \emph{string group generated by involutions} or \emph{sggi}. If $\G$ is an sggi, then the Cayley graph of $\G$ with respect to $\{x_0, \ldots, x_{n-1}\}$ is a regular premaniplex $\calM$ such that $\G(\calM) = \G$. In fact, this will be a maniplex if and only if each $x_i$ and each $x_i x_j$ (with $i \neq j$) is nontrivial. Since we can recover a regular premaniplex from its automorphism group, we often just work with sggis if we are interested in regular maniplexes or polytopes.

Consider the sggi $\G = \langle \rho_0, \ldots, \rho_{n-1} \rangle$. We say that $\G$ is a \emph{string C-group} if it satisfies the following \emph{intersection condition}:
\[ \forall I,J \subseteq \{0, \ldots, n-1\}, \langle \rho_i \mid i \in I \rangle \cap \langle \rho_j \mid j \in J \rangle = \langle \rho_k \mid k \in I \cap J \rangle. \]
The regular maniplex $\calM$ with automorphism group $\G$ is a polytope if and only if $\G$ is a string C-group.

Since polytopes are often defined inductively by first defining their facets and vertex-figures, the following result is frequently useful:

\begin{proposition}[{{\cite[Proposition 2E16(a)]{arp}}}]
\label{ind-string-c}
Suppose $\G = \langle \rho_0, \ldots, \rho_{n-1} \rangle$ is an sggi. Then $\G$ is a string C-group if and only if
\begin{enumerate}
    \item $\langle \rho_0, \ldots, \rho_{n-2} \rangle$ is a string C-group,
    \item $\langle \rho_1, \ldots, \rho_{n-1} \rangle$ is a string C-group, and
    \item $\langle \rho_0, \ldots, \rho_{n-2} \rangle \cap \langle \rho_1, \ldots, \rho_{n-1} \rangle = \langle \rho_1, \ldots, \rho_{n-2} \rangle$.
\end{enumerate}
\end{proposition}

Said another way, if we want to determine whether $\calM$ is a polytope, and we already know that its facets and vertex-figures are polytopes, then we only need to check one part of the intersection condition. Compare to \cref{recursive-pip}. 

There is another convenient way to represent the automorphism group of a regular premaniplex. Suppose that $\calM$ is a regular premaniplex with base flag $\Phi$. Consider the action of $W_n$ on the flags of $\Phi$ as described in \cref{sec:maniplexes}, and let $N = \Stab_{W_n}(\Phi)$. Then $\G(\calM) \cong \Conn(\calM) = W_n/N$.

More generally, if $\calM$ is any $n$-premaniplex with base flag $\Phi$ and $N = \Stab_{W_n}(\Phi)$, then $\G(\calM, \Phi) \cong V/N$, where $V = \Norm_{W_n}(N)$. The isomorphism $\G(\calM, \Phi) \cong V/N$ can be understood as follows.
Given $\alpha\in \G(\calM, \Phi) $, one can write $\Phi\alpha = v\Phi$, with $v\in V$; and vice versa, if $v\in V$, then there exists $\alpha_v\in \G(\calM, \Phi) $ such that $v\Phi=\Phi\alpha_v$. Moreover, if $v,w\in V$ with $vw^{-1}\in N$, then $v\Phi=w\Phi$, so that $\alpha_v=\alpha_w$.
Note that if we pick a different base flag $\Phi'$ and set $N' = \Stab_{W_n}(\Phi')$, then $N'$ is conjugate to $N$ in $W_n$; say $N' = N^w$. Furthermore, if $V' = \Norm_{W_n}(N')$, then $V'$ is conjugated to $V$ in the same way: $V' = V^w$.

If $\calM$ and $\calN$ are $n$-premaniplexes, we say that $\calM$ \emph{covers} $\calN$ if there is a surjective function $\varphi: \calM \to \calN$ that preserves edge colors, and then we call $\varphi$ a \emph{covering}. Under our assumption that premaniplexes are connected, a covering is uniquely determined by the image of any single flag. Note that if $\calM$ has base flag $\Phi$ and $\calN$ has base flag $\Psi$, then $\calM$ covers $\calN$ if and only if $\Stab_{W_n}(\Phi) \leq \Stab_{W_n}(\Psi)$ (see Lemma 2.1 of \cite{sparse}).
Furthermore, we can ``identify'' $\Stab_{W_n}(\Phi)$ with $\calM$, in the sense that the Schreier graph of $W_n$ with respect to $\Stab_{W_n}(\Phi)$ and the set $\{r_0, r_1,\dots, r_{n-1}\}$ is precisely $\calM$ (see Lemma 2.2 in \cite{sparse}).

Of course, regular premaniplexes are the most symmetric ones. 
In particular, for every $I\subseteq\{0,\dots, n-1\}$, their automorphism group acts transitively on the set of chains of type $I$. 
Some classes of non-regular premaniplexes nevertheless have many transitivity properties.
If the automorphism group of an $n$-premaniplex $\calM$ acts transitively on chains of type $I$, we shall say that $\calM$ is $I$-chain-transitive; if $I=\{i\}$, we say that $\calM$ is $i$-face-transitive.
Moreover, if $\calM$ is $i$-face-transitive for all $i\in \{0,\dots, n-1\}$, then we say that $\calM$ is {\em fully transitive}.

When $\calM$ is medial-section-transitive (that is, $I$-chain-transitive for $I = \{1, 2, \ldots, n-2\}$), we can make a further refinement of \cref{recursive-pip}. 

    \begin{proposition} \label{recursive-pip2}
    Let $\calM$ be an $n$-maniplex with base flag $\Phi$, and suppose that $\calM$ is medial-section-transitive. Then $\calM$ is polytopal if and only if
    \begin{enumerate}
        \item The facets are polytopal,
        \item The vertex figures are polytopal, and
        \item For every flag $\Psi$, if $\Phi \ijpath{0}{n-2} \Psi$ and $\Phi \ijpath{1}{n-1} \Psi$, then $\Phi \ijpath{1}{n-2} \Psi$. 
    \end{enumerate}
    \end{proposition}

    \begin{proof}
    The conditions of \cref{recursive-pip} imply these conditions, so it remains to prove the reverse implication. Let $\Lambda$ and $\Psi$ be arbitrary flags such that $\Lambda \ijpath{0}{n-2} \Psi$ and $\Lambda \ijpath{1}{n-1} \Psi$. Since $\calM$ is medial-section-transitive, there must be a flag $\Lambda'$ such that $\Lambda' \ijpath{1}{n-2} \Lambda$ and $\Lambda' \alpha = \Phi$ for some automorphism $\alpha$. Then 
    \[ \Lambda' \ijpath{1}{n-2} \Lambda \ijpath{0}{n-2} \Psi, \]
    and so $\Lambda' \ijpath{0}{n-2} \Psi$. Similarly, $\Lambda' \ijpath{1}{n-1} \Psi$. Since automorphisms commute with connections, acting on $\Lambda'$ and $\Psi$ by $\alpha$ gives us that $\Phi \ijpath{0}{n-2} \Psi \alpha$ and $\Phi \ijpath{1}{n-1} \Psi \alpha$, and so by hypothesis, $\Phi \ijpath{1}{n-2} \Psi \alpha$. Now acting by $\alpha^{-1}$ gives us $\Lambda' \ijpath{1}{n-2} \Psi$, and since $\Lambda \ijpath{1}{n-2} \Lambda'$, it follows that $\Lambda \ijpath{1}{n-2} \Psi$.
    \end{proof}

\subsection{Chiral and two-orbit maniplexes}
\label{sec:2-orbit}

    After regular premaniplexes, the next most symmetric structures are those where the automorphism group has two orbits on the flags; these are called \emph{two-orbit premaniplexes}. For each two-orbit premaniplex $\calM$, there is an $I \subset \{0, \ldots, n-1\}$ such that every flag $\Phi$ is in the same orbit as $\Phi^i$ for every $i \in I$; we then say that $\calM$ is in \emph{class $2_I$} (or $2_I^n$ if we also wish to emphasize the rank). Note that $I$ is a proper subset of $\{0, \ldots, n-1\}$, since if $\Phi^i$ is in the same orbit as $\Phi$ for every $i$, then the premaniplex is regular.

    We define $\textbf{2}_I^n$ to be the premaniplex consisting of two flags joined by edges of colors $\{0, \ldots, n-1\} \setminus I$, and with semi-edges of color $i$ at each node for each $i \in I$. If $\calM$ is in class $2_I^n$, then its symmetry type graph is precisely $\textbf{2}_I^n$.

    An $n$-premaniplex is \emph{chiral} if it is in class $2_{\emptyset}^n$, which we will also denote simply as $2^n$. Chiral premaniplexes and polytopes have been extensively studied in the past 20 years, and many tools have been developed to study them. Some of these tools, such as chirality groups (defined in \cref{sec:var-gps}) are easily adapted for use on other classes of two-orbit premaniplexes. 

    Similar to the case with regular premaniplexes, we do not take great care to distinguish between different base flags in the same orbit. However, which of the two orbits the base flag belongs to is important. If $\calM$ is a two-orbit premaniplex with base flag $\Phi$, then we use $\overline{\calM}$ to denote $\calM$ but with a base flag chosen in the orbit opposite of $\Phi$. It is convenient to extend this notation to regular premaniplexes by defining $\overline{\calM} = \calM$; one could think of $\overline{\calM}$ as having a base flag that is adjacent to the base flag of $\calM$, but for regular premaniplexes these two choices are equivalent.

    Let $\calM$ be an $n$-premaniplex in class $2^n_I$. 
    Then, the quotient of $\calM$ by $\G(\calM)$ is precisely $\textbf{2}_I^n$.
    Moreover, following \cite{isa2orbit}, one can then see that, if $|I|<n-1$, then  $\calM$ is fully transitive; 
    if $I=\{0,1,\dots, n-1\} \setminus \{j\}$, (so that $|I|=n-1$), then $\calM$ is  $i$-face-transitive, for all $i\neq j$.

    A premaniplex is called \emph{orientable} if it is bipartite. Equivalently, an $n$-premaniplex is orientable if and only if it covers $\textbf{2}^n$. More generally, we will say that $\calM$ is \emph{$I$-orientable} if it covers $\textbf{2}_I^n$, which then induces a bicoloring of the flags that is compatible with $\textbf{2}_I^n$. See also \cite{flag-bicolorings}, where the term \emph{$I$-colorable} is used rather than $I$-orientable. A premaniplex $\calM$ that is not $I$-orientable has a minimal $I$-orientable cover, called the \emph{$I$-double} of $\calM$, with twice as many flags as $\calM$ (see \cite{flag-bicolorings}). 

    A premaniplex is \emph{rotary} if it is either regular or chiral. We will call a premaniplex \emph{$I$-rotary} if it is either regular or in class $2_I$. A premaniplex that is both $I$-rotary and $I$-orientable will be called \emph{$I$-admissible}. 
    Thus, if $\calM$ is $I$-admissible, then it is either in class $2_I$ or it is regular and $I$-orientable.

    If $\calM$ is $I$-orientable, and $\alpha \in \G(\calM)$, then consider the action on the base flag:
    \[ \Phi \alpha = w \Phi = r_{i_1} \cdots r_{i_k} \Phi. \]
    If $|\{ j : i_j \not \in I \}|$ is even, then we will say that $\alpha$
    is \emph{$I$-even}; otherwise it is \emph{$I$-odd} -- this is well-defined precisely when $\calM$ is $I$-orientable. We define $\G^I(\calM)$ to be those automorphisms of $\calM$ that are $I$-even. When $I = \emptyset$, this is typically denoted $\G^+(\calM)$ instead.
    Note that since $\calM$ is $I$-orientable, its flags can be colored with two colors such that every closed walk traverses an even number of edges with labels not in $I$, and thus $\G^I(\calM)$ consists of all the automorphisms of $\calM$ that preserve this coloring. 
    In particular this means that if $\calM$ is a two-orbit premaniplex in class $2_I$, then $\G(\calM)=\G^I(\calM)$. We also define $\G^I(\calM) = \G(\calM)$ whenever $\calM$ is not $I$-orientable.

    For each proper subset $I \subset \{0, \ldots, n-1\}$, we define the following subgroup of $W = \langle r_0, \ldots, r_{n-1} \rangle$:
    \[ W_n^I = \langle \{r_i\}_{i \in I}, \{r_i r_j\}_{i,j \not \in I}, \{r_i r_j r_i\}_{j = i \pm 1 \in I, i \not \in I} \rangle. \]
    Then $W_n^I$ is an index 2 subgroup of $W_n$. Furthermore, $\alpha$ is $I$-even if and only if $\Phi \alpha = w \Phi$ with $w \in W_n^I$. 
    
    Suppose $\calM$ is an $n$-premaniplex with base flag $\Phi$, and let $N = \Stab_{W_n}(\Phi)$. If $\calM$ is $I$-admissible, 
    then $N \lhd W_n^I$, and then $\G^I(\calM) = W_n^I / N$. 
    When $\calM$ is also regular, $\G^I(\calM)$ is an index 2 subgroup of $\G(\calM)$. As with regular premaniplexes, it is possible to build an $I$-admissible premaniplex $\calM$ from a group $\G = W_n^I / N$ in such a way that $\G^I(\calM) = \G$, so that in principle, one can work with $I$-admissible premaniplexes entirely in terms of their automorphism groups.

    For each $I \subset \{0, \ldots, n-1\}$ we define the following subsets of $\{0, \ldots, n-2\}$: 
    \begin{eqnarray}
        I_{n-1} &=& I \setminus \{n-1\},  \nonumber \\
        I_0 &=& \{i -1 \mid i \geq 1, i \in I\}.\nonumber
    \end{eqnarray}
  %  $I_{n-1} = I \setminus \{n-1\}$, and 
 %   \[ I_0 = \{i -1 \mid i \geq 1, i \in I\}. \]
    That is, $I_0$ removes $0$ if it is present, and then shifts everything down by 1. These new subsets help us describe the following well-known facts about the facets and vertex-figures of $I$-orientable and $I$-rotary premaniplexes.

    \begin{proposition}
    \label{prop:i-or-facets}
    If $\calM$ is an $n$-premaniplex that is $I$-orientable (resp. $I$-rotary, $I$-admissible), then its facets are $I_{n-1}$-orientable (resp. $I_{n-1}$-rotary, $I_{n-1}$-admissible) and its vertex-figures are $I_0$-orientable (resp. $I_0$-rotary, $I_0$-admissible).
    \end{proposition}

    The $(n-2)$-faces of a chiral $n$-premaniplex must be regular. Some other classes of two-orbit premaniplexes exhibit this same phenomenon.

    \begin{proposition}
    \label{prop:n-2-reg}
    If $\calM$ is an $n$-premaniplex in class $2_I$ and $n-1 \not \in I$, then the $(n-2)$-faces of $\calM$ are regular.
    \end{proposition}

    \begin{proof}
    Consider an $(n-2)$-face $F$ and flags $\Phi$ and $\Psi$ in that face. We want to show that there is an automorphism sending $\Phi$ to $\Psi$. These flags naturally extend to flags $\tilde{\Phi}$ and $\tilde{\Psi}$ of $\calM$. If these flags are in the same orbit in $\calM$, say $\tilde{\Psi} = \tilde{\Phi} \gamma$, then the restriction of $\gamma$ to $F$ will send $\Phi$ to $\Psi$. Otherwise, by assumption there is some $j \not \in I$, $j \neq n-1$. Then $\tilde{\Psi}^{n-1}$ is in the same orbit as $\tilde{\Phi}$, say $\tilde{\Psi}^{n-1} = \tilde{\Phi} \gamma$. Then since $\tilde{\Psi}^{n-1}$ has the same $(n-2)$-face $F$ as $\Psi$, the restriction of $\gamma$ to $F$ will again send $\Phi$ to $\Psi$. 
    \end{proof}

 Note that the only class of two-orbit premaniplexes that are not facet-transitive are those in class $2_{\{0,1,\dots, n-2\}}^n$. 
 For premaniplexes in that class, their facets are in fact regular (of potentially two different isomorphism types). 
 Dually, premaniplexes in class $2_{\{1,2,\dots, n-1\}}^n$ have regular vertex figures.

 Note that if a premaniplex $\calM$ is facet-transitive and its facets are vertex-transitive, then $\calM$ is medial-section-transitive.
 Dually, if $\calM$ is vertex-transitive and its vertex figures are facet-transitive, then $\calM$ is medial-section-transitive.
 Hence, an $n$-premaniplex in class $2_I^n$ is medial-section-transitive whenever:
 \begin{itemize}
     \item  $|I|<n-2$,
     \item $|I|=n-2$ but $I \neq \{1,2,\dots, n-2\}$,
     \item $|I|=n-1$ with $\bar{I}\neq \{0\}$ or $\bar{I}\neq \{n-1\}$. 
 \end{itemize}

 For each $I\subset\{0,\dots, n-1\}$ let us now define $I_{0,n-1}$ as follows
 $$
I_{0,n-1} = \{i-1 \mid 1\leq i \leq n-2, i\in I\}.
 $$
Note that $I_{0,n-1}= (I_0)_{n-2}=(I_{n-1})_0$.

Using this notation, an $I$-rotary $n$-premaniplex is medial-section-transitive if and only if $I_{0,n-1}\neq \{0,2,\dots, n-3\}$.

 If $\calM$ is facet-transitive with base facet $\calK$, then the base facet of $\overline{\calM}$ is $\overline{\calK}$. On the other hand, if $\calM$ has two orbits of facets, then the base facet of $\calM$ is a representative from one orbit, while the base facet of $\overline{\calM}$ is a representative from the other orbit. Analogous results hold for the vertex-figures and medial sections of $\overline{\calM}$. 

\section{Mixing maniplexes}
\label{sec:mixing}

    Given two maniplexes, it is often useful to find the smallest maniplex that covers both of them. For example, every maniplex $\calM$ has a minimal self-dual cover that can be obtained by finding the smallest maniplex that covers $\calM$ and its dual. Such a minimal common cover can be constructed using the \emph{mix} of two maniplexes. Let us describe the fundamentals here; more details can be found in \cite[Sec. 5]{mixing-and-monodromy} and \cite[Sec. 3.2]{k-orbit}. (Note that although both of those references discuss the mix of \emph{polytopes}, the proofs usually do not use polytopality in any way, and so the results remain true for maniplexes.)

    First, let us consider two regular $n$-premaniplexes $\calM_1$ and $\calM_2$. Their automorphism groups are sggis; say $\G(\calM_1) = \langle \rho_0, \ldots, \rho_{n-1} \rangle$ and $\G(\calM_2) = \langle \rho_0', \ldots, \rho_{n-1}' \rangle$. Then the \emph{parallel product} or \emph{mix} of those groups, denoted $\G(\calM_1) \mix \G(\calM_2)$, is the subgroup of $\G(\calM_1) \times \G(\calM_2)$ generated by the elements of the form $(\rho_i, \rho_i')$ for each $i$ \cite{parallel-product}. That is,
    \[ \G(\calM_1) \mix \G(\calM_2) = \langle (\rho_0, \rho_0'), \ldots, (\rho_{n-1}, \rho_{n-1}') \rangle. \]
    Note that if $\G(\calM_1)$ and $\G(\calM_2)$ have identical relations, then $\G(\calM_1) \mix \G(\calM_2) \cong \G(\calM_1)$. In fact, if $\calM_1$ covers $\calM_2$, then $\calM_1 \mix \calM_2 \cong \calM_1$. At the other extreme, if $\G(\calM_1)$ and $\G(\calM_2)$ are ``different enough'', then it can happen that $\G(\calM_1) \mix \G(\calM_2) = \G(\calM_1) \times \G(\calM_2)$.

    Writing $\alpha_i$ for $(\rho_i, \rho_i')$, note that if the order of $\rho_{i_1} \cdots \rho_{i_t}$ is $k$ and the order of $\rho_{i_1}' \cdots \rho_{i_t}'$ is $m$, then the order of $\alpha_{i_1} \cdots \alpha_{i_t}$ is $\lcm(m,k)$. 
    In particular, this implies that the mix of two sggis is itself a sggi. 
    So we can build a regular premaniplex from $\G(\calM_1) \mix \G(\calM_2)$ using the correspondence mentioned in \cref{sec:reg-mans}, and we call this premaniplex $\calM_1 \mix \calM_2$, the \emph{mix} of $\calM_1$ and $\calM_2$. 
    Thus, by construction, $\G(\calM_1 \mix \calM_2) \cong \G(\calM_1) \mix \G(\calM_2)$.

    It is also possible to define the mix of two general premaniplexes with no symmetry conditions, in a way that is consistent with the above. Now, however, we must choose a base flag in each premaniplex. Given two rooted $n$-premaniplexes $(\calM_1, \Phi_1)$ and $(\calM_2, \Phi_2)$, consider the graph with vertex set $V(\calM_1) \times V(\calM_2)$, and where $(\Psi_1, \Psi_2)^i = (\Psi_1^i, \Psi_2^i)$. This graph may not be connected, so we define their mix $(\calM_1, \Phi_1) \mix (\calM_2, \Phi_2)$ to be the connected component that contains $(\Phi_1, \Phi_2)$.

    We note the following simple result.

    \begin{proposition} \label{mix-with-pre}
    The mix of a maniplex with a premaniplex is a maniplex.
    \end{proposition}

    \begin{proof}
    If the mix is not a maniplex, then it must have either a semi-edge or multiple edges. A semi-edge in the mix induces a semi-edge in each component, and multiple edges in the mix induce multiple edges in each component as well.
    \end{proof}

    % Moved up a bit: It is straightforward to see that if $\calM$ covers $\calN$, then $\calM\mix\calN\cong\calM$.

    As an application of \cref{mix-with-pre}, consider the mix $\calM \mix \textbf{2}_I^n$. If $\calM$ is $I$-orientable, then it covers $\textbf{2}_I^n$, and $\calM \mix \textbf{2}_I^n \cong \calM$. Otherwise, $\calM \mix \textbf{2}_I^n$ is the $I$-double of $\calM$ as mentioned before: the minimal $I$-orientable maniplex that covers $\calM$. 

     The next result says that in order to check the Path Intersection Property for the mix of a polytope with a premaniplex, it is enough to check that it holds for pairs of flags that agree in one coordinate. %Furthermore, if $\calP_1$ is a $k$-orbit polytope, then it's enough to pick one representative from each flag-orbit, and then check the property for each of those $k$ choices of first coordinate.

    \begin{proposition} \label{weaker-pip}
    Let $\calP_1$ be an $n$-polytope and let $\calP_2$ be an $n$-premaniplex. Then $\calP_1 \mix \calP_2$ is a polytope if and only if:
    \begin{enumerate}
    \item Its facets are polytopes,
    \item Its vertex-figures are polytopes, and
    \item For every flag $(\Psi_1, \Psi_2)$ of $\calP_1 \mix \calP_2$ and every flag $\Lambda_2$ of $\calP_2$, whenever $(\Psi_1, \Psi_2) \ijpath{0}{n-2} (\Psi_1, \Lambda_2)$ and $(\Psi_1, \Psi_2) \ijpath{1}{n-1} (\Psi_1, \Lambda_2)$, then $(\Psi_1, \Psi_2) \ijpath{1}{n-2} (\Psi_1, \Lambda_2)$.
    \end{enumerate}
    \end{proposition}
    
    \begin{proof}
    Let $\calQ = \calP_1 \mix \calP_2$. If $\calQ$ is a polytope then it satisfies the Path Intersection Property, which implies the given property.
    
    Conversely, suppose that $\calQ$ satisfies the given properties. 
    By \cref{mix-with-pre}, $\calQ$ is a maniplex. 
    In order to show that $\calQ$ is a polytope, it suffices to show that $\calQ$ satisfies the third condition of \cref{recursive-pip}. 
    Consider flags $(\Psi_1, \Psi_2)$ and $(\Lambda_1, \Lambda_2)$ such that $(\Psi_1, \Psi_2) \ijpath{0}{n-2} (\Lambda_1, \Lambda_2)$ and $(\Psi_1, \Psi_2) \ijpath{1}{n-1} (\Lambda_1, \Lambda_2)$. (See \cref{fig:weaker-pip}.)
    Then $\Psi_1 \ijpath{0}{n-2} \Lambda_1$ and $\Psi_1 \ijpath{1}{n-1} \Lambda_1$, and since $\calP_1$ is a polytope, it follows by \cref{pip} that $\Psi_1 \ijpath{1}{n-2} \Lambda_1$. That is, $\Psi_1 = w \Lambda_1$, with $w \in \langle r_1, \ldots, r_{n-2} \rangle \leq W_n$. Thus $(\Lambda_1, \Lambda_2) \ijpath{1}{n-2} (\Psi_1, w \Lambda_2)$. It follows that $(\Psi_1, \Psi_2) \ijpath{0}{n-2} (\Psi_1, w\Lambda_2)$ and $(\Psi_1, \Psi_2) \ijpath{1}{n-1} (\Psi_1, w\Lambda_2)$, and so by hypothesis, $(\Psi_1, \Psi_2) \ijpath{1}{n-2} (\Psi_1, w\Lambda_2)$. 
    Therefore $(\Psi_1, \Psi_2) \ijpath{1}{n-2} (\Lambda_1, \Lambda_2)$, and the result follows.
    \end{proof}

    \begin{figure}
        \centering
    \begin{tikzpicture}
    % Nodes
    \node (A) at (0,0) {$(\Psi_1, \Psi_2)$};
    \node (B) at (4,0) {$(\Lambda_1, \Lambda_2)$};
    \node (C) at (8,0) {$(\Psi_1, w\Lambda_2)$};
    
    % Edges
    \draw[-, bend left=45] (A) to node[above] {$[0, n-2]$} (B);
    \draw[-, bend right=45] (A) to node[below] {$[1, n-1]$} (B);
    \draw[-] (B) to node[above] {$[1, n-2]$} (C);
    \draw[-, dashed, bend right=60] (A) to node[below] {$[1, n-2]$} (C);
    \end{tikzpicture}
    \caption{In $\calP_1 \mix \calP_2$, the full Path Intersection Property follows from the special case where we only consider flags that agree in the first coordinate. See \cref{weaker-pip}.}
    \label{fig:weaker-pip}
    \end{figure}
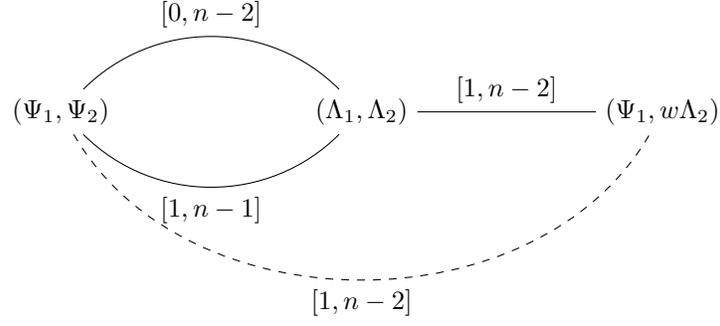

\section{$\calT$-admissibility}
\label{sec:t-admissibility}

    Consider the action of $W_n$ on the premaniplexes $(\calM_1, \Phi_1)$ and $(\calM_2, \Phi_2)$, and let $N_i = \Stab_{W_n}(\Phi_i)$ for $i \in \{1,2\}$. Then, considering the action of $W_n$ on $(\calM_1, \Phi_1) \mix (\calM_2, \Phi_2)$, we see that $\Stab_{W_n}(\Phi_1, \Phi_2) = N_1 \cap N_2$. In particular, if $\calM_1$ and $\calM_2$ are both regular, so that $N_1$ and $N_2$ are normal, then $N_1 \cap N_2$ is normal, and $\G(\calM_1 \mix \calM_2) \cong W_n / (N_1 \cap N_2)$.
    In general, followint the discussion of Section~\ref{sec:reg-mans}, the fact that $\Stab_{W_n}(\Phi_1,\Phi_2)=N_1\cap N_2$ implies that $\Norm_{W_n}(N_1\cap N_2)/(N_1\cap N_2)$ can be regarded as the automorphism group of $\calM_1\mix\calM_2$.

    Now let us consider two $I$-admissible premaniplexes $(\calM_1, \Phi_1)$ and $(\calM_2, \Phi_2)$. We can follow a strategy analogous to what we did for regular premaniplexes: Find $\G^I(\calM_1) \mix \G^I(\calM_2)$ and then build an $I$-admissible premaniplex from the result. In this case, the flag-stabilizers $N_1$ and $N_2$ are normal in $W_n^I$ but not necessarily in $W_n$, and their intersection $N_1 \cap N_2$ is similarly normal in $W_n^I$ but not necessarily in $W_n$. This definition of the mix is again consistent with the combinatorial definition of the mix. For $k$-orbit premaniplexes with $k > 2$, the preceding approach is not always workable, essentially because the normalizers of the flag-stabilizers may not be equal or comparable under inclusion.

    In fact, though $\G(\calM_1 \mix \calM_2) \leq \G(\calM_1) \times \G(\calM_2)$ for \emph{regular} premaniplexes $\calM_1$ and $\calM_2$, this containment is not true in general. For example, if $\calM_1$ and $\calM_2$ are chiral but their mix is regular, then $\G(\calM_1 \mix \calM_2)$ will contain an automorphism that sends the base flag $(\Phi_1, \Phi_2)$ to $(\Phi_1^0, \Phi_2^0)$, and this does not induce an automorphism of the two coordinates since $\calM_1$ and $\calM_2$ are chiral and so $\Phi_i^0$ is in a different orbit from $\Phi_i$. However, it is possible to describe a subgroup of $\G(\calM_1 \mix \calM_2)$ that is contained in $\G(\calM_1) \times \G(\calM_2)$.

    \begin{proposition}
    \label{mix-subdirect}
    Let $(\calT, \Psi)$ be an $n$-premaniplex, and let $N = \Stab_{W_n}(\Psi)$. Consider $n$-maniplexes $(\calM_i, \Phi_i)$ with $i \in \{1,2\}$, and let $N_i = \Stab_{W_n}(\Phi_i)$. Suppose that $N_i \leq N \leq \Norm_{W_n}(N_i)$ for $i \in \{1,2\}$. Then:
    \begin{enumerate}
        \item $N / (N_1 \cap N_2) \leq \G(\calM_1 \mix \calM_2)$.
        \item $N / (N_1 \cap N_2) \leq \G(\calM_1) \times \G(\calM_2)$.
    \end{enumerate}        
    \end{proposition}

    \begin{proof}
    For the first part, since $N$ normalizes $N_1$ and $N_2$, it also normalizes $N_1 \cap N_2$. So $N \leq \Norm_{W_n}(N_1 \cap N_2)$, and since $N_1$ and $N_2$ are contained in $N$, so is $N_1 \cap N_2$, which implies that $N / (N_1 \cap N_2) \leq \Norm_{W_n}(N_1 \cap N_2) / (N_1 \cap N_2) = \G(\calM_1 \mix \calM_2)$. 

    For the second part, first note that $N / (N_1 \cap N_2)$ can be identified with a subgroup of $N/N_1 \times N/N_2$ via the homomorphism that sends $g(N_1 \cap N_2) \mapsto (gN_1, gN_2)$. This is well-defined since if $g(N_1 \cap N_2) = h(N_1 \cap N_2)$, then $h^{-1} g \in N_1 \cap N_2$, and thus $g N_1 = h N_1$ and $g N_2 = h N_2$. Then since $N / N_i$ is a subgroup of $\Norm_{W_n}(N_i) / N_i = \G(\calM_i)$, the result follows.
    \end{proof}

    Let us make the following definition, which is similar in spirit to the definition of $I$-admissible and to \cite[Definition 3.1]{map-operations}: 

    \begin{definition}
    If $(\calM, \Phi)$ and $(\calT, \Psi)$ are $n$-premaniplexes, then $\calM$ is \emph{$\calT$-admissible} if
    \[ \Stab_{W_n}(\Phi) \leq \Stab_{W_n}(\Psi) \leq \Norm_{W_n}(\Stab_{W_n}(\Phi)).\] 
    If we wish to emphasize the base flags, we may also say that $(\calM, \Phi)$ is $(\calT, \Psi)$-admissible.
    \end{definition}

Note that being $I$-admissible is equivalent to being $\bf{2}_I$-admissible.
Furthermore, one can see from the definition that $\calM$ is $\calT$-admissible if and only if $\calT$ is a symmetry type graph of $\calM$ with respect to some $\Gamma\leq \Gamma(\calM)$.
   
    With this definition, the conditions of \cref{mix-subdirect} can be rephrased as saying that $\calM_1$ and $\calM_2$ are both $\calT$-admissible.

     \begin{corollary}
     \label{coro:mixadmissible}
        If $\calM_1$ and $\calM_2$ are $\calT$-admissible $n$-premaniplexes, %then $(\calM_i,\Phi_i)$ and $(\calT,\Phi)$ are $n$-premaniplexes and $\calM_i$ is $\calT$-admissible, with $i \in \{1,2\}$, 
        then $\calM_1\mix\calM_2$ is also $\calT$-admissible.
    \end{corollary}

    \begin{proof}
    Following the notation of \cref{mix-subdirect}, its proof implies that $N_1 \cap N_2 \leq N \leq \Norm_{W_n}(N_1 \cap N_2)$. 
    \end{proof}

    % Here is another way to think about $\calT$-admissibility. The fact that $\Stab(\Phi) \leq \Stab(\Psi)$ means that $(\calM, \Phi)$ covers $(\calT, \Psi)$; that is, the function that sends each $w\Phi \mapsto w\Psi$ (ranging over all $w \in W_n$) is well-defined. Then, the fact that $\Stab(\Psi) \leq \Norm(\Stab(\Phi))$ means that if $\Phi' \mapsto \Psi$, then $\Phi$ is in the same flag-orbit as $\Phi$, since an element $w \in W_n$ determines an automorphism of $\calM$ if and only if $w \in \Norm(\Stab(\Phi))$ (see \cref{sec:reg-mans}). 

    The next result generalizes \cref{prop:i-or-facets}.

    \begin{proposition}
    \label{prop:t-adm-facets}
    If $(\calM, \Phi)$ is $(\calT, \Psi)$-admissible, then the facet of $\calM$ containing $\Phi$ is $(\calT_{n-1}, \Psi)$-admissible, where $(\calT_{n-1}, \Psi)$ is the connected component containing $\Psi$ of the graph obtained from $\calT$ by deleting edges of label $n-1$.
    \end{proposition}

    \begin{proof}
    The covering from $\calM$ to $\calT$ induces a covering from the base facet $F$ of $\calM$ to the base facet of $\calT$. Suppose $\Phi'$ is a flag in the base facet of $\calM$ such that $\Phi' \mapsto \Psi$. Since $\calM$ is $\calT$-admissible, $\calT$ is a symmetry type graph of $\calM$ and thus there is a $\gamma \in \G(\calM)$ such that $\Phi' = \Phi \gamma$. Then $\gamma$ fixes the base facet of $\calM$ and thus induces an automorphism of that facet, which means that $\Phi'$ is in the same $\G(F)$-orbit as $\Phi$. Thus, $F$ is $G$-admissible, where $G$ is the base facet of $\calT$, which is precisely $\calT_{n-1}$ as defined.
    \end{proof}

    \begin{lemma}
    \label{lem:t-adm-coords}
    Suppose that $\calM_1$ and $\calM_2$ are $\calT$-admissible. If $(\Psi_1, \Psi_2)$ and $(\Psi_1', \Psi_2)$ are flags of $\calM_1 \mix \calM_2$, %there is a path from $(\Psi_1, \Psi_2)$ to $(\Psi_1', \Psi_2)$ in $\calM_1 \mix \calM_2$, 
    then $\Psi_1' = \Psi_1 \gamma$ for some $\gamma \in \G(\calM_1)$.
    \end{lemma}

    \begin{proof}
    Let $\Psi$ be the base flag of $\calT$. If $(\Psi_1, \Psi_2)$ and $(\Psi_1', \Psi_2)$ are both flags of $\calM_1 \mix \calM_2$, then there is a path between them, meaning that 
    %The existence of a path from $(\Psi_1, \Psi_2)$ to $(\Psi_1', \Psi_2)$ means that 
    there is some $w \in W_n$ such that $w \Psi_1 = \Psi_1'$ and $w \Psi_2 = \Psi_2$. Thus, $w \in \Stab_{W_n}(\Psi_2)$, and by $\calT$-admissibility,
    \[ \Stab_{W_n}(\Psi_2) \leq \Stab_{W_n}(\Psi) \leq \Norm_{W_n}(\Stab_{W_n}(\Psi_1)). \]
    Thus, $w \in \Norm_{W_n}(\Stab_{W_n}(\Psi_1))$, which means that there is $\gamma \in \G(\calM_1)$ such that $w \Psi_1 = \Psi_1 \gamma$.
    \end{proof}

    Let us briefly discuss the action of $\G(\calM_1 \mix \calM_2)$ on faces of $\calM_1 \mix \calM_2$. Again, our experience from mixing regular maniplexes may lead us astray. When $\calM_1$ and $\calM_2$ are regular, their automorphism groups act transitively on faces of any given rank, and indeed on chains of any given type. Their mix $\calM_1 \mix \calM_2$ is also regular and shares these transitivity properties. However, in general, the mix of two premaniplexes may lose some transitivity properties that the components had. For example, consider \cref{fig:med-sect-trans}, which shows the mix of the $3$-orbit $4$-maniplexes $3^{1,2}$ and $3^2$ (with semi-edges suppressed). Both $3^{1,2}$ and $3^2$ are medial-section-transitive; removing the edges of color $0$ and $3$ leaves each of them connected. However, removing edges of color $0$ and $3$ from their mix gives us two connected components, showing that the mix is not medial-section-transitive.

    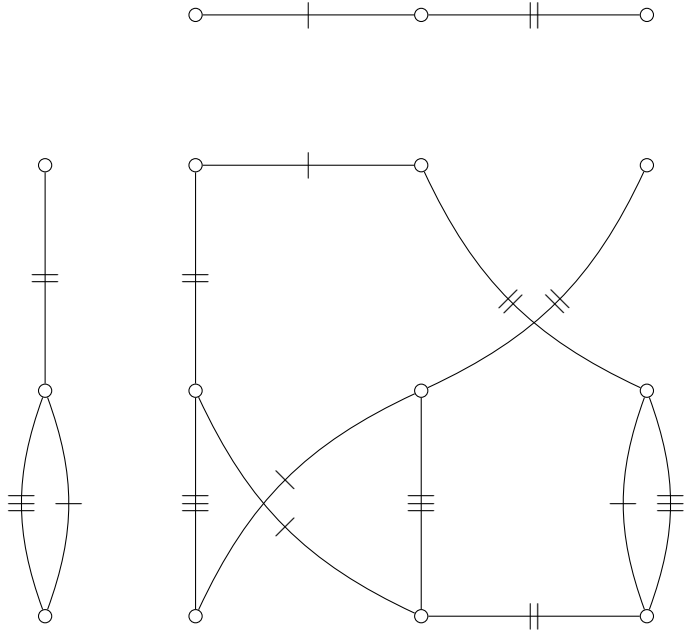
\begin{figure}
        \centering
\begin{tikzpicture}
    % Nodes
    \node[circle,draw=black, fill=white, inner sep=0pt,minimum size=5pt] (A) at (3,8) {};
    \node[circle,draw=black, fill=white, inner sep=0pt,minimum size=5pt] (B) at (6,8) {};
    \node[circle,draw=black, fill=white, inner sep=0pt,minimum size=5pt] (C) at (9,8) {};

    \node[circle,draw=black, fill=white, inner sep=0pt,minimum size=5pt] (1) at (1,6) {};
    \node[circle,draw=black, fill=white, inner sep=0pt,minimum size=5pt] (2) at (1,3) {};
    \node[circle,draw=black, fill=white, inner sep=0pt,minimum size=5pt] (3) at (1,0) {};

    \node[circle,draw=black, fill=white, inner sep=0pt,minimum size=5pt] (A1) at (3,6) {};
    \node[circle,draw=black, fill=white, inner sep=0pt,minimum size=5pt] (A2) at (3,3) {};
    \node[circle,draw=black, fill=white, inner sep=0pt,minimum size=5pt] (A3) at (3,0) {};
    
    \node[circle,draw=black, fill=white, inner sep=0pt,minimum size=5pt] (B1) at (6,6) {};
    \node[circle,draw=black, fill=white, inner sep=0pt,minimum size=5pt] (B2) at (6,3) {};
    \node[circle,draw=black, fill=white, inner sep=0pt,minimum size=5pt] (B3) at (6,0) {};

    \node[circle,draw=black, fill=white, inner sep=0pt,minimum size=5pt] (C1) at (9,6) {};
    \node[circle,draw=black, fill=white, inner sep=0pt,minimum size=5pt] (C2) at (9,3) {};
    \node[circle,draw=black, fill=white, inner sep=0pt,minimum size=5pt] (C3) at (9,0) {};
    
    % Edges
    \draw[-] (A) to node  {$|$} (B);
    \draw[-] (B) to node {$||$} (C);

    \draw[-] (1) to node[sloped] {$||$} (2);
    \draw[-, bend left = 20] (2) to node[sloped] {$|$} (3);
    \draw[-, bend right = 20] (2) to node[sloped] {$|||$} (3);

    \draw[-] (A1) to node[sloped] {$|$} (B1);
    \draw[-, bend right=20] (A2) to node[sloped] {$|$} (B3);
    \draw[-, bend left=20] (A3) to node[sloped] {$|$} (B2);
    \draw[-, bend right=20] (C2) to node[sloped] {$|$} (C3);

    \draw[-] (A1) to node[sloped] {$||$} (A2);
    \draw[-, bend right=20] (B1) to node[sloped] {$||$} (C2);
    \draw[-, bend left=20] (C1) to node[sloped] {$||$} (B2);
    \draw[-] (B3) to node[sloped] {$||$} (C3);

    \draw[-] (A2) to node[sloped] {$|||$} (A3);
    \draw[-] (B2) to node[sloped] {$|||$} (B3);
    \draw[-, bend left=20] (C2) to node[sloped] {$|||$} (C3);
    \end{tikzpicture}
    \caption{The mix of medial-section-transitive premaniplexes need not be medial-section-transitive. Here we show the mix of the $3$-orbit $4$-maniplexes $3^{1,2}$ and $3^2$, with semi-edges suppressed; an edge with label $i$ is denoted with $i$ hash marks. Both inputs remain connected when we remove edges with label $0$ and $3$, but the mix (consisting of the nine vertices in the lower-right corner) does not.}
        \label{fig:med-sect-trans}
    \end{figure}

    While we can make no promises about transitivity in general, we can use \cref{coro:mixadmissible} to generalize what happens for regular maniplexes.    

    \begin{proposition}
    \label{lem:med-sect-mix}
    Let $I\subseteq\{0,\dots, n-1\}$.
    Suppose that $\calT$ is an $n$-premaniplex such that it has only one connected component when removing the edges of color $i\in I$. 
    If $\calM_1$ and $\calM_2$ are $\calT$-admissible, then $\G(\calM_1 \mix \calM_2)$ acts transitively on chains of $\calM_1 \mix \calM_2$ of  type $I$.
    \end{proposition}

    \begin{proof}
Since $\calM_i$ is $\calT$-admissible, by \cref{coro:mixadmissible} so is $\calM_1\mix\calM_2$. 
Hence, $\calT$ is a symmetry type graph of $\calM_1\mix\calM_2$.  Since $\calT$  has only one connected component when removing the edges of color $i\in I$, that means that $\calM_1\mix\calM_2$ is transitive on chains of type $I$.
    \end{proof}

\section{Chirality groups and variance groups}
\label{sec:var-gps}

    Suppose $(\calM, \Phi)$ is a chiral $n$-premaniplex, with flag-stabilizer $N$. Then $N$ is not normal in $W_n$; indeed, $N^{r_i} = N^{r_j} \neq N$ for each $i$ and $j$. The \emph{smallest regular cover} of $\calM$ is the regular premaniplex $\calR$ with flag stabilizer $N \cap N^{r_0}$; every regular premaniplex that covers $\calM$ also covers $\calR$. It is possible to measure `how chiral' $\calM$ is by considering how different $N$ is from $N^{r_0}$. There are two ways of doing that that in some sense coincide, but let us be precise. We define the \emph{upper chirality group of $\calM$} to be
    \[ X_{\wedge}(\calM) := N^{r_0} / (N \cap N^{r_0}) \leq \Conn(\calM) \]
    and the \emph{lower chirality group of $\calM$} to be
    \[ X_{\vee}(\calM) := N N^{r_0} / N \leq \G^+(\calM). \]
    These are isomorphic as abstract groups, and if we only care about the abstract group type, then we call that the \emph{chirality group of $\calM$}, denoted $X(\calM)$. In practice, it usually suffices to work with just the upper chirality group or just the lower chirality group, and so $X(\calM)$ may refer to either $X_{\vee}(\calM)$ or $X_{\wedge}(\calM)$ if it is clear from context which is meant.

    Note the that smallest regular cover of a chiral premaniplex $(\calM, \Phi)$ is precisely the mix $(\calM, \Phi) \mix (\calM, \Phi^0)$ (or, written another way, $\calM \mix \overline{\calM}$). This suggests that the chirality group could be generalized to tell us `how different' two rooted premaniplexes are with respect to their mix. This was done in \cite{var-gps} for chiral and orientable regular polytopes, and we present it here for premaniplexes with no symmetry restrictions, with some small tweaks.

    Suppose that $(\calM_1, \Phi_1)$ and $(\calM_2, \Phi_2)$ are premaniplexes with flag stabilizers $N_1$ and $N_2$, respectively. %Then in analogy with chirality groups, we can try to measure how different $\calM_1$ and $\calM_2$ are. The notion will have to be a bit more nuanced now, because there is a symmetry present in the definition of the chirality group (namely, that $N$ and $N^{r_0}$ are conjugate), and the natural generalization breaks that symmetry. 
    The natural thing to try is to simply modify the definition of the chirality groups by changing $N$ to $N_1$ and $N^{r_0}$ to $N_2$. We have to take some care, however; it is not guaranteed that $N_1 \cap N_2$ is normal in $N_2$ or that $N_1$ is normal in $N_1 N_2$. With that in mind, we define the \emph{upper variance group of $(\calM_1, \Phi_2)$ with respect to $(\calM_2, \Phi_2)$} to be
    \[ \UVG{\calM_1, \Phi_1}{\calM_2, \Phi_2} := N_2 / (N_1 \cap N_2) \] %\leq \G((\calM_1, \Phi_1) \mix (\calM_2, \Phi_2)), \]
    whenever this is well-defined (i.e., whenever $N_1 \cap N_2$ is normal in $N_2$). We define the \emph{lower variance group of $(\calM_1, \Phi_1)$ with respect to $(\calM_2, \Phi_2)$} as 
    \[ \LVG{\calM_1, \Phi_1}{\calM_2, \Phi_2} := N_1 N_2 / N_1 \] %\leq \G(\calM_1, \Phi_1), \]
    whenever this is well-defined (i.e., whenever $N_1$ is normal in $N_1 N_2$). As abstract groups, $\UVG{\calM_1, \Phi_1}{\calM_2, \Phi_2} \cong \LVG{\calM_1, \Phi_1}{\calM_2, \Phi_2}$, and so if we only care about the isomorphism type we may refer to this simply as the \emph{variance group of $(\calM_1, \Phi_1)$ with respect to $(\calM_2, \Phi_2)$}, denoted $\VG{\calM_1, \Phi_1}{\calM_2, \Phi_2}$. Note that if $\calM$ is chiral, then
    \[ X_{\wedge}(\calM) = X_{\wedge}(\calM, \Phi | \calM, \Phi^0) \]
    and
    \[ X_{\vee}(\calM) = X_{\vee}(\calM, \Phi | \calM, \Phi^0). \]
    We will frequently abbreviate our notation by omitting the base flags, writing $\LVG{\calM_1}{\calM_2}$ and $\UVG{\calM_1}{\calM_2}$.

    Let us discuss some cases that ensure that the variance groups are well-defined.

    \begin{proposition}
    \label{var-gp-well-defd}
    If $(\calM_1, \Phi_1)$ and $(\calM_2, \Phi_2)$ are $(\calT, \Psi)$-admissible for some $(\calT, \Psi)$, then $\LVG{\calM_1, \Phi_1}{\calM_2, \Phi_2}$ and $\UVG{\calM_1, \Phi_1}{\calM_2, \Phi_2}$ are well-defined.
    \end{proposition}

    \begin{proof}
    let $V = \Stab_{W_n}(\Psi)$, and let $N_i = \Stab_{W_n}(\Phi_i)$. Then $N_i \leq V \leq \Norm(N_i)$ for $i \in \{1,2\}$, and thus $N_1 N_2 \leq V \leq \Norm(N_1)$. Thus, $N_1$ is normal in $N_1 N_2$. Similarly, we have
    \[ N_2 \leq V \leq \Norm(N_1) \cap \Norm(N_2) \leq \Norm(N_1 \cap N_2), \]
    and so $N_1 \cap N_2$ is normal in $N_2$.     
    \end{proof}

    Note that the structure of $\calT$ plays no role in \cref{var-gp-well-defd}; we merely need for some such $\calT$ to exist. 

    Let us now provide a more combinatorial interpretation of $\LVG{\calM_1}{\calM_2}$.

    \begin{proposition}
    \label{var-gp-coordinate}
    Suppose $\LVG{\calM_1}{\calM_2}$ is well-defined.  %and let $\gamma \in \G(\calM_1)$. 
    % Then $\gamma \in \LVG{\calM_1}{\calM_2}$ if and only if 
    Then 
    \[ \LVGM = \{ \gamma \in \G(\calM_1) : (\Phi_1 \gamma, \Phi_2) \in (\calM_1, \Phi_1) \mix (\calM_2, \Phi_2) \}.\]
    \end{proposition}

    \begin{proof}
    By the definition of the mix, $(\Phi_1 \gamma, \Phi_2)$ is in $(\calM_1, \Phi_1) \mix (\calM_2, \Phi_2)$ if and only if there is some $w \in W$ such that $(w \Phi_1, w \Phi_2) = (\Phi_1 \gamma, \Phi_2)$. Since $w$ fixes $\Phi_2$, this is true if and only if there is $w \in N_2$ such that $\Phi_1 \gamma = w \Phi_1$. Let us show that this is true if and only if there is $w' \in N_1 N_2$ such that $\Phi_1 \gamma = w' \Phi_1$. Clearly the existence of $w$ implies the existence of $w'$. Conversely, given $w' = uv$ with $u \in N_1$ and $v \in N_2$, then $uvu^{-1}$ is an element of $N_2$, because $N_1$ must normalize $N_2$ if $N_1 N_2 / N_1$ is well-defined. Furthermore, $uvu^{-1} \Phi_1 = uv \Phi_1 = w' \Phi_1$, so taking $w' = uv u^{-1}$ has the desired properties. Finally, note that such a $w'$ exists if and only if $\gamma \in N_1 N_2 / N_1 = \LVG{\calM_1}{\calM_2}$. 
    \end{proof}

    Thus, we can think of \cref{var-gp-coordinate} as saying that $\LVG{\calM_1}{\calM_2}$ measures to what extent we can act independently on the coordinates of $\calM_1 \mix \calM_2$ with automorphisms.
    
    To get some intuition for these definitions, let us consider some special cases.

    \begin{example}
    Suppose that $(\calM_1, \Phi_1)$ is covered by $(\calM_2, \Phi_2)$, so that $N_2 \leq N_1$. 
    Then $N_1\cap N_2=N_2$ and $N_1N_2 = N_1$, so $\UVG{\calM_1, \Phi_1}{\calM_2, \Phi_2}$ and $\LVG{\calM_1, \Phi_1}{\calM_2, \Phi_2}$ are both trivial. 
    \end{example}

    In fact, the following result is immediate:

    \begin{proposition} \label{trivial-var-gp}
    $\VG{\calM_1, \Phi_1}{\calM_2, \Phi_2}$ is trivial if and only if $(\calM_2, \Phi_2)$ covers $(\calM_1, \Phi_1)$.
    \end{proposition}

    \begin{example}
    Now suppose that $(\calM_1, \Phi_1)$ covers $(\calM_2, \Phi_2)$, so that $N_1 \leq N_2$. In this case, $(\calM_1, \Phi_1) \mix (\calM_2, \Phi_2) = (\calM_1, \Phi_1)$, and both $\UVG{\calM_1, \Phi_1}{\calM_2, \Phi_2}$ and $\LVG{\calM_1, \Phi_1}{\calM_2, \Phi_2}$ are isomorphic to $N_2 / N_1$, provided that this is well-defined.
    \end{example}

    In particular, note that if $\calM_1$ and $\calM_2$ are regular, and $\calM_1$ covers $\calM_2$, then $\VG{\calM_1}{\calM_2}$ is the kernel of the covering from $\G(\calM_1)$ to $\G(\calM_2)$.

    \begin{example}
    \label{ex:var-gp-polygons}
    Suppose that $\calM_1$ is a $p$-gon and $\calM_2$ is a $q$-gon. Considering the action of $W_2$ on these maniplexes, we have $N_1 = \langle (r_0 r_1)^p \rangle$ and $N_2 = \langle (r_0 r_1)^q \rangle$. Thus $N_1 \cap N_2 = \langle (r_0 r_1)^\ell \rangle$, where $\ell = \lcm(p,q)$, and $N_1 N_2 = \langle (r_0 r_1)^g \rangle$, where $g = \gcd(p,q)$. Writing $\G(\calM_1) = \langle \rho_0, \rho_1 \rangle$, we get that $\LVG{\calM_1}{\calM_2}$ can be identified with the subgroup of $\G(\calM_1)$ generated by $(\rho_0 \rho_1)^g$, of order $p/g$. Similarly, if $\G(\calM_1 \mix \calM_2) = \langle \rho_0', \rho_1' \rangle$, then $\UVG{\calM_1}{\calM_2}$ can be identified with the subgroup of $\G(\calM_1 \mix \calM_2)$ generated by $(\rho_0' \rho_1')^q$, of order $\ell/q = p/g$. 
    \end{example}

    Now let us consider the general case. The upper variance group $\UVG{\calM_1}{\calM_2}$ in some sense measures how much you have to add to $\G(\calM_2)$ to get a cover of $\calM_1$. In fact, if $\calM_1$ and $\calM_2$ are regular, then $\UVGM$ is the kernel of the projection from $\G(\calM_1 \mix \calM_2)$ to $\G(\calM_2)$. The lower variance group $\LVG{\calM_1}{\calM_2}$ is a subgroup of $\Norm_{W_n}(N_1) / N_1$, which is the automorphism group of $\calM_1$. It essentially measures how much of $\calM_1$ you have to collapse in order to get a quotient of $\calM_2$. 
    If $\calM_1$ and $\calM_2$ are regular, then $\LVGM$ is the kernel of the projection from $\G(\calM_1)$ to %$\G(\calM_1) \comix \G(\calM_2)$, 
    the largest common quotient of $\G(\calM_1)$ and $\G(\calM_2)$.

    Thus, $\VG{\calM_1}{\calM_2}$ will be small if and only if $\calM_1$ is `almost' covered by $\calM_2$, and it will be `big' (that is, have low index in $\G(\calM_1)$) if and only if $\calM_1$ is `far away' from being a quotient of $\calM_2$ (i.e., if we have to collapse most of $\calM_1$ in order to get a quotient of $\calM_2$).

\section{When is the mix of $\calT$-admissible maniplexes a polytope?}
\label{sec:mix-i-admissible}

\subsection{Determining polytopality using variance groups}

Our goal in this section is to take the criteria of \cref{weaker-pip} and translate them to algebraic criteria. 
Throughout this section, we will continue to suppress the choice of base flag to streamline the presentation, but the reader should keep in mind that different choices of base flag can yield materially different results for variance groups and when mixing. Furthermore, we will only be dealing with lower variance groups from now on, and we will drop the subscript $\vee$.

The results in this section often require that two premaniplexes both be $\calT$-admissible for some $\calT$. Note that this is true in particular if both premaniplexes are regular, or if they are both $I$-admissible for some $I$ (in particular if they are both in class $2_I$).

    % For our first result, note that if premaniplexes $\calM_1$ and $\calM_2$ have facets $\calK_1$ and $\calK_2$, respectively, then we can naturally consider $\VGK$ as a subgroup of $\G(\calM_1)$, since 

    \begin{proposition} \label{var-gp-as-paths-3}
    Let $\calM_1$ and $\calM_2$ be $\calT$-admissible $n$-premaniplexes for some $\calT$. For $i \in \{1,2\}$, consider a flag $(\Psi_1, \Psi_2)$ in $\calM_1 \mix \calM_2$ such that the facets, vertex-figures, and medial sections of $\Psi_i$ are isomorphic to $\calK_i$, $\calL_i$, and $\calN_i$, respectively.  
    Then the groups $\VGK, \VGL$, and $\VGN$ are well-defined. Furthermore, %, and if $\gamma \in \G(\calM_1)$, then: 
    \begin{enumerate}
        \item $\VGK$ can be identified with 
        \[ \{\gamma \in \G(\calM_1) : (\Psi_1, \Psi_2) \ijpath{0}{n-2} (\Psi_1 \gamma, \Psi_2) \text{ in } \calM_1 \mix \calM_2 \} \]
        %if and only if $(\Psi_1, \Psi_2) \ijpath{0}{n-2} (\Psi_1 \gamma, \Psi_2)$ in $\calM_1 \mix \calM_2$.
        \item $\VGL$ can be identified with
        \[ \{\gamma \in \G(\calM_1) : (\Psi_1, \Psi_2) \ijpath{1}{n-1} (\Psi_1 \gamma, \Psi_2) \text{ in } \calM_1 \mix \calM_2 \} \]
%        $\gamma \in \VGL$ if and only if $(\Psi_1, \Psi_2) \ijpath{1}{n-1} (\Psi_1 \gamma, \Psi_2)$ in $\calM_1 \mix \calM_2$.
        \item $\VGN$ can be identified with
        \[ \{\gamma \in \G(\calM_1) : (\Psi_1, \Psi_2) \ijpath{1}{n-2} (\Psi_1 \gamma, \Psi_2) \text{ in } \calM_1 \mix \calM_2 \} \] 
%        $\gamma \in \VGN$ if and only if $(\Psi_1, \Psi_2) \ijpath{1}{n-2} (\Psi_1 \gamma, \Psi_2)$ in $\calM_1 \mix \calM_2$.
    \end{enumerate}
    \end{proposition}

    \begin{proof}
    %\gabe{I am treating $\VGK$ as a subgroup of $\G(\calM_1)$ here, and I need to explain how. If $\calM_1$ is hereditary, then that is OK: $\VGK \leq \G(\calK_1) \leq \G(\calM_1)$. Otherwise??} 
    We will prove the first result; the others follow similarly. 
    By \cref{prop:t-adm-facets}, $\calK_1$ and $\calK_2$ are both $\calT_{n-1}$-admissible, where $\calT_{n-1}$ is the base facet of $\calT$. It follows that $\VGK$ is well-defined, by \cref{var-gp-well-defd}. 
    
    Now, suppose $\gamma \in \VGK$, viewed as a subgroup of $\G(\calK_1)$. Then by \cref{var-gp-coordinate}, $(\Psi_1 \gamma, \Psi_2) \in (\calK_1, \Psi_1) \mix (\calK_2, \Psi_2)$, which implies that $(\Psi_1, \Psi_2) \ijpath{0}{n-2} (\Psi_1 \gamma, \Psi_2)$ in $\calM_1 \mix \calM_2$. Then \cref{lem:t-adm-coords} proves that $\Psi_1 \gamma = \Psi_1 \alpha$ for some $\alpha \in \G(\calM_1)$, and since automorphisms act semiregularly on flags, this implies that $\alpha = \gamma$ so that $\gamma \in \G(\calM_1)$.

    Conversely, suppose that $(\Psi_1, \Psi_2) \ijpath{0}{n-2} (\Psi_1 \gamma, \Psi_2)$ for some $\gamma \in \G(\calM_1)$. Then $\gamma$ fixes the facet of $\Psi_1$, so its restriction to that facet is an automorphism of $\G(\calK_1)$. By semiregularity, no other automorphism of $\calM_1$ will have the same restriction to $\calK_1$, and so identifying $\gamma$ with its restriction, we get that $\gamma \in \VGK$ by \cref{var-gp-coordinate}. 
    \end{proof}

    In \cref{thm:all-i-admissible-mix}, we will characterize when the mix of a polytope with a premaniplex is a polytope. We want the polytope to come first, since that is our primary interest, but then it will turn out that we care about variance groups such as $X(\calK_2 | \calK_1)$ rather than $X(\calK_1 | \calK_2)$. Since $\calP_2 \mix \calP_1$ is naturally isomorphic to $\calP_1 \mix \calP_2$ by switching coordinates,  \cref{var-gp-as-paths-3} implies that $\gamma \in X(\calK_2 | \calK_1)$ if and only if 
    \[ (\Psi_1, \Psi_2) \ijpath{0}{n-2} (\Psi_1, \Psi_2 \gamma), \]
    with similar modifications for $\VGLb$ and $X(\calN_{2} | \calN_{1})$.

    \begin{proposition}
    \label{polytopal-implies-vargp}
    Suppose that $\calM_1$ and $\calM_2$ are $\calT$-admissible $n$-premaniplexes for some $\calT$, and that $\calM_1 \mix \calM_2$ is a polytope. Let $(\Psi_1, \Psi_2)$ be a flag of $\calM_1 \mix \calM_2$ such that the facets, vertex-figures, and medial sections of $\Psi_i$ are $\calK_i$, $\calL_i$, and $\calN_i$, respectively. Then
    \[ X(\calK_2 | \calK_1) \cap X(\calL_2 | \calL_1) \leq X(\calN_2 | \calN_1). \]
    \end{proposition}

    \begin{proof}
    Let $\gamma \in X(\calK_2 | \calK_1) \cap X(\calL_2 | \calL_1)$. By \cref{var-gp-as-paths-3} (and switching coordinates as described in the previous paragraph), that means that $(\Psi_1, \Psi_2) \ijpath{0}{n-2} (\Psi_1, \Psi_2 \gamma)$ and also $(\Psi_1, \Psi_2) \ijpath{1}{n-1} (\Psi_1, \Psi_2 \gamma)$. Then by \cref{pip}, $(\Psi_1, \Psi_2) \ijpath{1}{n-2} (\Psi_1, \Psi_2 \gamma)$, which again by \cref{var-gp-as-paths-3} implies that $\gamma \in X(\calN_2 | \calN_1)$.
    \end{proof}

    \begin{theorem}
    \label{thm:all-i-admissible-mix}
    Suppose that $\calP_1$ is an $n$-polytope and $\calP_2$ is an $n$-premaniplex, with both $\calP_1$ and $\calP_2$ $\calT$-admissible for some $\calT$. Then $\calP_1 \mix \calP_2$ is a polytope if and only if:
    \begin{enumerate}
        \item The facets of $\calP_1 \mix \calP_2$ are polytopes,
        \item The vertex-figures of $\calP_1 \mix \calP_2$ are polytopes, and
        \item For every flag $(\Psi_1, \Psi_2)$ of $\calP_1 \mix \calP_2$, if we let the facets, vertex-figures, and medial sections of $\Psi_i$ be $\calK_i$, $\calL_i$, and $\calN_i$, respectively, then
        \[ X(\calK_2 | \calK_1) \cap X(\calL_2 | \calL_1) \leq X(\calN_2 | \calN_1). \]
    \end{enumerate}
    \end{theorem}

    \begin{proof}
    If $\calP_1 \mix \calP_2$ is a polytope, then certainly its facets and vertex-figures are polytopes, and \cref{polytopal-implies-vargp} proves the last condition.

    Conversely, suppose that $\calP_1 \mix \calP_2$ satisfies the conditions above. We need only prove that $\calP_1 \mix \calP_2$ satisfies the third condition of \cref{weaker-pip}. 
    Consider a flag $(\Psi_1, \Psi_2)$ and fix a flag $\Lambda_2$ of $\calP_2$ such that $(\Psi_1, \Psi_2) \ijpath{0}{n-2} (\Psi_1, \Lambda_2)$ and $(\Psi_1, \Psi_2) \ijpath{1}{n-1} (\Psi_1, \Lambda_2)$. 
    Since $\calP_1$ and $\calP_2$ are $\calT$-admissible, by \cref{lem:t-adm-coords} there exists $\gamma\in\G(\calM_2)$ such that $\Lambda_2=\Psi_2\gamma$. %$I$-orientable and $(\Psi_1, \Psi_2)$ is connected to $(\Psi_1, \Lambda_2)$, this implies that $\Lambda_2 = w \Psi_2$ for some connection $w$ that is $I$-even. Then since  $\calP_2$ is $I$-rotary, there is an $I$-even automorphism $\gamma$ such that $\Lambda_2 = \Psi_2 \gamma$. 
    Then \cref{var-gp-as-paths-3} implies that $\gamma \in X(\calK_2 | \calK_1) \cap X(\calL_2 | \calL_1)$. Then by assumption, $\gamma \in X(\calN_2 | \calN_1)$, which by \cref{var-gp-as-paths-3} implies that $(\Psi_1, \Psi_2) \ijpath{1}{n-2} (\Psi_1, \Psi_2 \gamma) = (\Psi_1, \Lambda_2)$, which is what we needed to show.
    \end{proof}

      \begin{corollary}
    \label{trivial-intersection}
    Suppose that $\calP_1$ is an $n$-polytope and $\calP_2$ is an $n$-premaniplex, with both $\calP_1$ and $\calP_2$ $\calT$-admissible, for some premaniplex $\calT$. 
    Suppose that $\calP_1$ and $\calP_2$ are facet- and vertex-transitive, and suppose that $\calP_i$ has facets isomorphic to $\calK_i$ and vertex-figures isomorphic to $\calL_i$ for $i \in \{1,2\}$. If $\calK_1 \mix \calK_2$ and $\calL_1 \mix \calL_2$ are polytopes and $\VGKb \cap \VGLb$ is trivial, then $\calP_1 \mix \calP_2$ is a polytope.      
      \end{corollary}

    The next result is closely related to \cite[Lemma 3.3 and Remark 3.4]{chiral-mix}. 

\begin{proposition} \label{facets-cover-polytopal}
Let $\calP_1$ be an $n$-polytope and  $\calP_2$ be an $n$-premaniplex. 
Suppose that $\calP_1$ and $\calP_2$ are both $\calT$-admissible. 
If, for each flag $(\Psi_1, \Psi_2)$ of $\calP_1 \mix \calP_2$, the facet of $\Psi_1$ covers that of $\Psi_2$, then $\calP_1 \mix \calP_2$ is a polytope.
\end{proposition}

\begin{proof}
First, note that $n = 2$ is true, since in this case $\calP_1$ and $\calP_2$ are both polygons and so is their mix. 

Now, suppose that the result holds in rank $n-1$. Consider a flag $(\Psi_1, \Psi_2)$ and let $\calK_i$ be the facet of $\Psi_i$. Since $\calK_1$ covers $\calK_2$, \cref{trivial-var-gp}  implies that $X(\calK_{2} | \calK_{1})$ is trivial, so the third condition of \cref{thm:all-i-admissible-mix} is automatically satisfied. Similarly, the facet of $(\Psi_1, \Psi_2)$ is $\calK_1 \mix \calK_2 = \calK_1$, which is polytopal. It remains to show that the vertex-figure of $(\Psi_1, \Psi_2)$ is polytopal.

Let $\calL_i$ be the vertex-figure of $\Psi_i$. Then the vertex-figure of $(\Psi_1, \Psi_2)$ is isomorphic to $\calL_1 \mix \calL_2$. Note that since $\calK_1$ covers $\calK_2$, that implies that the vertex-figures of $\calK_{1}$ cover those of $\calK_{2}$. In other words, the facets of $\calL_{1}$ cover those of $\calL_{2}$. Furthermore, each $\calL_i$ is $\calT_0$-admissible, and $\calL_1$ is a polytope. Thus, by inductive hypothesis,  $\calL_{1} \mix \calL_{2}$ is a polytope, completing the proof.
\end{proof}

In the case that $\calP_1$ and $\calP_2$ are medial-section-transitive, we get a particularly nice corollary of \cref{thm:all-i-admissible-mix}.

    \begin{corollary}
    \label{cor:med-sec-trans-I}
    Suppose that $\calP_1$ is an $n$-polytope and $\calP_2$ is an $n$-premaniplex, with both $\calP_1$ and $\calP_2$ $\calT$-admissible for some $\calT$. Suppose that $\calP_1$ and $\calP_2$ are medial-section-transitive. For $i \in \{1,2\}$, suppose that $\calP_i$ has facets, vertex-figures, and medial sections isomorphic to $\calK_i$, $\calL_i$, and $\calN_i$, respectively. Then $\calP_1 \mix \calP_2$ is a polytope if and only if:
    \begin{enumerate}
        \item $\calK_1 \mix \calK_2$ is a polytope,
        \item $\calL_1 \mix \calL_2$ is a polytope, and
        \item $X(\calK_2 | \calK_1) \cap X(\calL_2 | \calL_1) \leq X(\calN_2 | \calN_1)$.
    \end{enumerate}
    \end{corollary}

    Note that in particular, \cref{cor:med-sec-trans-I} applies when $\calP_1$ and $\calP_2$ are both regular or both $I$-admissible with $\overline{I} \not \in \{\{0\}, \{n-1\}, \{0,n-1\}\}$.

    With \cref{cor:med-sec-trans-I}, we can refine the third condition of \cref{weaker-pip} so that rather than having to consider every pair of flags that agree in the first coordinate, we may insist that the first flag be the base flag of the mix.

    \begin{theorem} \label{mix-pip2}
    Let $\calP_1$ be an $n$-polytope and let $\calP_2$ be an $n$-premaniplex, and let $\Phi_i$ be the base flag of $\calP_i$ for $i \in \{1,2\}$. Let $\calT$ be a medial-section-transitive premaniplex, and suppose that $\calP_1$ and $\calP_2$ are $\calT$-admissible. Let $\calK_i$ and $\calL_i$ be the facets and vertex-figures (respectively) of $\calP_i$, $i \in \{1,2\}$. Then $\calP_1 \mix \calP_2$ is a polytope if and only if
    \begin{enumerate}
        \item $\calK_1 \mix \calK_2$ is a polytope,
        \item $\calL_1 \mix \calL_2$ is a polytope,
        \item For every flag $\Lambda_2$ of $\calP_2$, whenever $(\Phi_1, \Phi_2) \ijpath{0}{n-2} (\Phi_1, \Lambda_2)$ and $(\Phi_1, \Phi_2) \ijpath{1}{n-1} (\Phi_1, \Lambda_2)$, then $(\Phi_1, \Phi_2) \ijpath{1}{n-2} (\Phi_1, \Lambda_2)$.
    \end{enumerate}
    \end{theorem}

    \begin{proof}
    If $\calP_1 \mix \calP_2$ is a polytope, then the conditions follow from \cref{weaker-pip}. 

    Conversely, suppose that the three conditions hold. We will show that the third condition of \cref{cor:med-sec-trans-I} holds. Let $\gamma \in \VGKb \cap \VGLb$. By \cref{var-gp-as-paths-3} (and considering the fact that the base flag $\Phi_i$ of $\calP_i$ has facets $\calK_i$ and vertex-figures $\calL_i$), it follows that $(\Phi_1, \Phi_2) \ijpath{0}{n-2} (\Phi_1, \Phi_2 \gamma)$ and $(\Phi_1, \Phi_2) \ijpath{1}{n-1} (\Phi_1, \Phi_2 \gamma)$. Then by assumption, $(\Phi_1, \Phi_2) \ijpath{1}{n-2} (\Phi_1, \Phi_2 \gamma)$, and then since the base flag has medial sections $\calN_i$, \cref{var-gp-as-paths-3} implies that $\gamma \in \VGNb$.
    \end{proof}

      In the following sections we showcase several applications of \cref{thm:all-i-admissible-mix} and its corollaries. In some cases, similar results have appeared in the literature before, though we can now generalize many of them.

\subsection{Mixes of $4$-polytopes}%\isa{here and above, maybe "Mixes of" is better than "Applications of \cref{thm:all-i-admissible-mix}" }

   %Let us start with the following corollary of \cref{cor:med-sec-trans-I}.
The mix of polyhedra is always a polyhedron (see \cite[Theorem 3.8]{k-orbit}). Let us use this and apply \cref{thm:all-i-admissible-mix} to the mix of $4$-polytopes. 

\begin{proposition} \label{mix-4-polytopes}
Suppose $\calP_1$ and $\calP_2$ are $\calT$-admissible, medial-section-transitive $4$-polytopes. Suppose that each $\calP_i$ has facets isomorphic to $\calK_i$, vertex-figures isomorphic to $\calL_i$, and medial sections isomorphic to $\calN_i$. Then $\calP_1 \mix \calP_2$ is a polytope if and only if $X(\calK_2 | \calK_1) \cap X(\calL_2 | \calL_1) \leq X(\calN_2 | \calN_1)$.
\end{proposition}

\begin{proof}
By \cite[Theorem 3.8]{k-orbit}, the mix of two polyhedra is always a polyhedron. Then the result follows from \cref{cor:med-sec-trans-I} since the facets and vertex-figures of $\calP_1 \mix \calP_2$ must be polytopes. 
\end{proof}

The next result is somewhat analogous to \cite[Thm. 9.1]{chiral-mix}.

\begin{proposition} \label{mix-4-orientable}
Suppose that $\calP_i$ is a regular $4$-polytope of type $\{p_i, q_i, r_i\}$ for $i \in \{1,2\}$. If $\calP_2$ is orientable, and if $\gcd(q_1, q_2) = 1$, then $\calP_1 \mix \calP_2$ is polytopal.
\end{proposition}

\begin{proof}
By \cref{mix-4-polytopes}, it suffices to show $X(\calK_2 | \calK_1) \cap X(\calL_2 | \calL_1) \leq X(\calN_2 | \calN_1)$. 
Let $\G(\calP_2) = \langle \rho_0, \ldots, \rho_3 \rangle$. Since $\gcd(q_1, q_2) = 1$, we have $X(\calN_2 | \calN_1) = \langle \rho_1 \rho_2 \rangle$ (see \cref{ex:var-gp-polygons}). Then, since $\calP_2$ is a polytope, we have $X(\calK_2 | \calK_1) \cap X(\calL_2 | \calL_1)$ will be contained in $\langle \rho_1, \rho_2 \rangle$, and since $\calP_2$ is orientable, $X(\calK_2 | \calK_1) \cap \langle \rho_1, \rho_2 \rangle = X(\calK_2 | \calK_1) \cap \langle \rho_1 \rho_2 \rangle$. Thus 
$X(\calK_2 | \calK_1) \cap X(\calL_2 | \calL_1) \leq X(\calN_2 | \calN_1)$.
\end{proof}

%Returning to general rank,  

\subsection{Mixes of $I$-admissible maniplexes that are not medial-section-transitive}
\label{sec:non-med-trans-mix}
    As pointed out before, \cref{cor:med-sec-trans-I} can be applied to $I$-admissible polytopes and premaniplexes, when $\bar{I} \notin \{ \{0\}, \{n-1\}, \{0,n-1\}\}$ 
    Let us now apply \cref{thm:all-i-admissible-mix} to the  mixes of $I$-admissible polytopes for $\bar{I} \in \{ \{0\}, \{n-1\}, \{0,n-1\}\}$. 
    That essentially amounts to considering the possible facets, vertex-figures, and medial-sections of each possible flag of $\calP_1 \mix \calP_2$. 
    Recall that $I$-admissible premaniplexes have at most two orbits of facets, vertex-figures and medial-sections. %, and note that if a premaniplex if medial-section-transitive, then it is also facet- and vertex-figure-transitive.

    \begin{corollary}
    \label{cor:2_0n-1}
    Suppose that $\calP_1$ is an $n$-polytope and $\calP_2$ is an $n$-premaniplex. Let $I = \{1, \ldots, n-2\}$, and suppose that $\calP_1$ and $\calP_2$ are $I$-admissible. For $i \in \{1,2\}$, suppose that $\calP_i$ has facets and vertex-figures isomorphic to $\calK_i$ and $\calL_i$, respectively. Further, suppose that the base medial section of $\calP_i$ is isomorphic to $\calN_{i,1}$ and that the other orbit of medial section is isomorphic to $\calN_{i,2}$. Then $\calP_1 \mix \calP_2$ is a polytope if and only if:
    \begin{enumerate}
        \item $\calK_1 \mix \calK_2$ is a polytope,
        \item $\calL_1 \mix \calL_2$ is a polytope, and
        \item $X(\calK_2 | \calK_1) \cap X(\calL_2 | \calL_1) \leq X(\calN_{2,1} | \calN_{1,1}) \cap X(\calN_{2,2} | \calN_{1,2})$.
    \end{enumerate}
    \end{corollary}

    \begin{corollary}
    \label{cor:2_n-1}
    Suppose that $\calP_1$ is an $n$-polytope and $\calP_2$ is an $n$-premaniplex. Let $I = \{0, 1, \ldots, n-2\}$, and suppose that $\calP_1$ and $\calP_2$ are $I$-admissible. For $i \in \{1,2\}$, suppose that $\calP_i$ has vertex-figures isomorphic to $\calL_i$. Suppose that the base facet (resp. medial section) of $\calP_i$ is isomorphic to $\calK_{i,1}$ (resp. $\calN_{i,1})$, and that the other orbit of facet (resp. medial section) is isomorphic to $\calK_{i,2}$ (resp. $\calN_{i,2}$). Then $\calP_1 \mix \calP_2$ is a polytope if and only if:
    \begin{enumerate}
        \item $\calK_{1,1} \mix \calK_{2,1}$ and $\calK_{1,2} \mix \calK_{2,2}$ are polytopes,
        \item $\calL_1 \mix \calL_2$ is a polytope, and
        \item $X(\calK_{2,j} | \calK_{1,j}) \cap X(\calL_2 | \calL_1) \leq X(\calN_{2,j} | \calN_{1,j})$ for $j \in \{1,2\}$.
    \end{enumerate}
    \end{corollary}
 
    \begin{corollary}
    \label{cor:2_0}
    Suppose that $\calP_1$ is an $n$-polytope and $\calP_2$ is an $n$-premaniplex. Let $I = \{1, \ldots, n-1\}$, and suppose that $\calP_1$ and $\calP_2$ are $I$-admissible. For $i \in \{1,2\}$, suppose that $\calP_i$ has facets isomorphic to $\calK_i$. Suppose that the base vertex-figure (resp. medial section) of $\calP_i$ is isomorphic to $\calL_{i,1}$ (resp. $\calN_{i,1})$, and that the other orbit of vertex-figure (resp. medial section) is isomorphic to $\calL_{i,2}$ (resp. $\calN_{i,2}$). Then $\calP_1 \mix \calP_2$ is a polytope if and only if:
    \begin{enumerate}
        \item $\calK_1 \mix \calK_2$ is a polytope,
        \item $\calL_{1,1} \mix \calL_{2,1}$ and $\calL_{1,2} \mix \calL_{2,2}$ are polytopes, and
        \item $X(\calK_{2} | \calK_{1}) \cap X(\calL_{2,j} | \calL_{1,j}) \leq X(\calN_{2,j} | \calN_{1,j})$ for $j \in \{1,2\}$.
    \end{enumerate}
    \end{corollary}

\subsection{Mixes of regular polytopes}

Regular polytopes are the most studied ones, and have the nice property that their automorphisms groups are easy to understand and handle. We present some discussion and results regarding the mix of regular polytopes. 
Let us start by expanding on \cite[Lemma 3.8]{faithful-thin}.

    \begin{proposition}
\label{i-double-poly}
Let $\calP$ be an $I$-non-orientable regular $n$-polytope. Suppose $\calP$ has facets $\calK$, vertex-figures $\calL$, and medial sections $\calN$, and let $\calQ$ be its $I$-double (see \cref{sec:2-orbit}). 
\begin{enumerate}[(a)]
    \item If $\calK$ is $I_{n-1}$-orientable or $\calL$ is $I_0$-orientable, then $\calQ$ is a polytope.
    \item If $\calK$ is $I_{n-1}$-non-orientable and $\calL$ is $I_0$-non-orientable, and if $\calN$ is $I_{0,n-1}$-orientable, then $\calQ$ is not a polytope.
    \item If $\calN$ is $I_{0,n-1}$-non-orientable, then $\calQ$ is a polytope if and only if the $I_{n-1}$-double of $\calK$ and the $I_0$-double of $\calL$ are polytopal.
\end{enumerate}
\end{proposition}

\begin{proof}
Let $\calP_1 = \calP$ and $\calP_2 = \textbf{2}_I^n$. Then $\calQ = \calP_1 \mix \calP_2$. The facets of $\calQ$ are $\calK' = \calK \mix \textbf{2}_{I_{n-1}}^{n-1}$, the vertex-figures are $\calL' = \calL \mix \textbf{2}_{I_0}^{n-1}$, and the medial sections are $\calN' = \calN \mix \textbf{2}_{I_{0,n-1}}^{n-2}$. Note that the premaniplex $\textbf{2}_I^n$ is regular, so that we can apply \cref{cor:med-sec-trans-I} to $\calQ$.

To prove part (a), first suppose that $\calK$ is $I_{n-1}$-orientable and $\calL$ is $I_0$-orientable. Then $\calK' = \calK$ and $\calL' = \calL$, so $\calQ$ has polytopal facets and vertex-figures. Furthermore, since $\calK$ is $I_{n-1}$-orientable, it follows that it covers $\textbf{2}_{I_{n-1}}^{n-1}$, and so by \cref{trivial-var-gp}, $X(\calK_2 | \calK_1)$ is trivial (since here $\calK_2 = \textbf{2}_{I_{n-1}}^{n-1}$ and $\calK_1 = \calK$). Then by \cref{trivial-intersection}, $\calQ$ is a polytope.

Now let us suppose that $\calK$ is $I_{n-1}$-orientable and $\calL$ is $I_0$-non-orientable. As in the previous case, the fact that $\calK$ is $I_{n-1}$-orientable implies that $X(\calK_2 | \calK_1)$ is trivial. The facets of $\calQ$ are isomorphic to $\calK$, which is a polytope. Thus, by \cref{trivial-intersection}, it only remains to show that the vertex-figure $\calL'$ of $\calQ$ is a polytope. This is clearly true if $n = 3$. Now suppose that part (a) is true up to rank $k$ and let $n = k+1$. Let $J = I_0$, so that $\calL$ is $J$-non-orientable. The fact that $\calK$ is $I_{n-1}$-orientable implies that the facets of $\calL$ are $I_{0,n-1} = J_{k-1}$-orientable. If the vertex-figures of $\calL$ are $J_0$-orientable, then the previous case proves that $\calL$ is polytopal. Otherwise, $\calL$ is a $k$-polytope that is $J$-non-orientable and with $J_{k-1}$-orientable facets and $J_0$-non-orientable vertex-figures, and by inductive hypothesis, $\calL$ is polytopal. So in any case, $\calL$ is a polytope, and so is $\calQ$. If we instead assume that $\calK$ is $I_{n-1}$-non-orientable and $\calL$ is $I_0$-orientable, then the results follows from a dual argument. 

Next we consider part (b). 
Similar to the cases above, if $\calN$ is $I_{0,n-1}$-orientable, then $X(\calN_2 | \calN_1)$ is trivial. If $\calK$ is $I_{n-1}$-non-orientable and $\calL$ is $I_0$-non-orientable, then $X(\calK_2 | \calK_1)$ is nontrivial, and thus is equal to $\G(\calP_2)$ of order 2. Similarly,  $X(\calL_2 | \calL_1) = \G(\calP_2)$, and so $X(\calK_2 | \calK_1) \cap X(\calL_2 | \calL_1)$ is nontrivial. Thus the third condition of \cref{cor:med-sec-trans-I} is violated, proving that $\calQ$ is not polytopal.

Finally, we prove part (c). If the medial sections of $\calP$ are $I_{0,n-1}$-non-orientable, then $X(\calN_2 | \calN_1)$ is nontrivial and thus all of $\G(\calP_2)$, and so the third condition of \cref{cor:med-sec-trans-I} is satisfied, giving the last result.
\end{proof}

Note that the above result cannot be easily generalized by replacing the regularity condition with a $\calT$-admissibility condition. This is because, in order to use the results in this section, we would need $\calP_2$ (which is $\textbf{2}_I^n$) to be $\calT$-admissible, but this is never the case unless $\calT = \textbf{2}_I^n$, or consists of a single vertex and $n$ semi-edges (in which case  $\calT$-admissibility is equivalent to regularity).

 Let us consider a concrete example to showcase how \cref{thm:all-i-admissible-mix} and its corollaries are useful. Suppose that $\calP_1$ is a regular $4$-polytope with facets $\calK_1$ of type $\{4,4\}$ and that $\calP_2$ is a regular $4$-polytope with facets $\calK_2$ of type $\{4,4\}$. (That is, the facets and vertex-figures both have quadrilateral facets and $4$-valent vertices.) By \cref{mix-4-polytopes}, we need only show that $\VGKb \cap \VGLb \leq \VGNb$. Let $\G(\calP_2) = \langle \rho_0, \rho_1, \rho_2, \rho_3 \rangle$. Then $\VGKb$ is a subgroup of $\G(\calP_2)$; in fact, it is a subgroup of $\langle \rho_0, \rho_1, \rho_2 \rangle$. Similarly, $\VGLb$ is a subgroup of $\langle \rho_1, \rho_2, \rho_3 \rangle$. Thus, $\VGKb \cap \VGLb$ is a subgroup of 
    \[ \langle \rho_0, \rho_1, \rho_2 \rangle \cap \langle \rho_1, \rho_2, \rho_3 \rangle, \]
    which equals $\langle \rho_1, \rho_2 \rangle$ since $\calP_2$ is a polytope. Now, every regular polyhedron of type $\{4,4\}$ covers the regular maniplex $\{4,4\}_{(1,0)}$. Then the order of the image of $\rho_1 \rho_2$ is still $4$ in this quotient, and it follows that no element of $\langle \rho_1, \rho_2 \rangle$ is in $\VGKb$. By the above discussion, this implies that $\VGKb \cap \VGLb$ is trivial, and so $\calP_1 \mix \calP_2$ is a polytope. A similar argument works if $\calP_1$ and $\calP_2$ are both rotary (chiral or regular) with any facets of type $\{4,4\}$. Note that we have not used any information about the vertex-figures of $\calP_1$ and $\calP_2$.

    More generally, if the facets $\calK_1$ and $\calK_2$ both cover some polytope $\calK_3$ with the same vertex-figures as $\calK_2$, then a similar argument works:

    \begin{proposition} \label{compatible-facets}
    Suppose that $\calP_1$ and $\calP_2$ are regular $n$-polytopes. Let $\calP_i$ have facets $\calK_i$ with $i \in \{1,2\}$. If there is a regular $(n-1)$-premaniplex $\calK_3$ such that $\calK_1$ and $\calK_2$ both cover $\calK_3$, and if the vertex-figures of $\calK_3$ are isomorphic to the vertex-figures of $\calK_2$, then $\calP_1 \mix \calP_2$ is a polytope.
    \end{proposition}
    
    \begin{proof}
    Let $\G(\calP_2) = \langle \rho_0, \ldots, \rho_{n-1} \rangle$. Then $\VGKb$ is a subgroup of $\langle \rho_0, \ldots, \rho_{n-2} \rangle$ and $\VGLb$ is a subgroup of $\langle \rho_1, \ldots, \rho_{n-1} \rangle$. Thus $\VGKb \cap \VGLb$ is a subgroup of
    \[ \langle \rho_0, \ldots, \rho_{n-2} \rangle \cap \langle \rho_1, \ldots, \rho_{n-1} \rangle, \]
    which is equal to $\langle \rho_1, \ldots, \rho_{n-2} \rangle$ since $\calP_2$ is a polytope. Now, $\VGKb$ consists of those elements of $\G(\calK_2)$ that we need to collapse in order to get a quotient of $\G(\calK_1)$, and since $\calK_2$ and $\calK_1$ both cover $\calK_3$, which has the same vertex-figures as $\calK_2$, that implies that no nontrivial element of $\langle \rho_1, \ldots, \rho_{n-2} \rangle$ is in $\VGKb$. Since $\VGKb \cap \VGLb$ is contained in $\langle \rho_1, \ldots, \rho_{n-2} \rangle$, this implies that $\VGKb \cap \VGLb$ is trivial, and thus contained in $\VGNb$.

    Now, the vertex-figures of $\calK_1$ cover those of $\calK_3$, which are the same as those of $\calK_2$. Since $\calK_1$ is a polytope (being the facet of a polytope), we can appeal to the dual of \cref{facets-cover-polytopal} to see that $\calK_1 \mix \calK_2$ is a polytope. Similarly, letting the vertex-figures of $\calP_i$ be $\calL_i$ for $i \in \{1,2\}$, we have that the facets of $\calL_1$ cover the facets of $\calL_2$, and so again by \cref{facets-cover-polytopal}, $\calL_1 \mix \calL_2$ is a polytope. Thus, $\calP_1 \mix \calP_2$ is a polytope, by \cref{cor:med-sec-trans-I}.
    \end{proof}

    In light of \cref{compatible-facets}, the example from before could be extended to higher rank by letting the facets both be cubic toroids (of type $\{4, 3^{n-2}, 4\}$), in which case they both cover $\{4, 3^{n-2}, 4\}_{(1, 0^{n-3})}$ (see \cite[Section 6D]{arp}). 

    We expect that an analogue of \cref{compatible-facets} holds for other classes of highly-symmetric polytopes, but the way we proved it for regular polytopes depended on knowing the appropriate intersection condition, so we do not develop that further here.

\subsection{Smallest regular covers of two-orbit polytopes}

Little is known about the smallest regular covers of two-orbit polytopes that are not chiral. Here we present a generalization of \cite[Theorem 3.1]{chiral-covers}. %mix of two-orbit polytopes that are not chiral.  Here we present some interesting results.  
    %, and we similarly extend this notation to $\calK$ and so on.%\gabe{Put this in the background somewhere}
 % First, let us generalize \cite[Theorem 3.1]{chiral-covers}. 
 (Recall that for a two-orbit premaniplex, we use $\overline{\calP}$ to mean $\calP$ but with a base flag in the other orbit.)
    \begin{corollary}
    \label{gen-src-poly}
    Let $\calP$ be an $n$-polytope in class $2_I$, with base facet, vertex-figure, and medial section isomorphic to $\calK$, $\calL$, and $\calN$, respectively. %Then the smallest regular cover of $\calP$ is a polytope if and only if $\calK \mix \overline{\calK}$ and $\calL \mix \overline{\calL}$ are both polytopes and
    %     $\VG{\overline{\calK}}{\calK} \cap \VG{\overline{\calL}}{\calL} \leq \VG{\overline{\calN}}{\calN}.$ 
    \begin{enumerate}
        \item If $\calP$ is medial-section-transitive, then the smallest regular cover of $\calP$ is a polytope if and only if $\calK \mix \overline{\calK}$ and $\calL \mix \overline{\calL}$ are both polytopes and
        $\VG{\overline{\calK}}{\calK} \cap \VG{\overline{\calL}}{\calL} \leq \VG{\overline{\calN}}{\calN}.$ 

        \item Suppose that $I = \{1, \ldots, n-2\}$. Let the base medial section of $\overline{\calP}$ be isomorphic to $\calN'$. Then the smallest regular cover of $\calP$ is a polytope if and only if $\calK \mix \overline{\calK}$ and $\calL \mix \overline{\calL}$ are both polytopes and $\VG{\overline{\calK}}{\calK} \cap \VG{\overline{\calL}}{\calL} \leq \VG{\calN}{\calN'} \cap \VG{\calN'}{\calN}.$ 

        \item Suppose that $I = \{0, 1, \ldots, n-2\}$. Let the base facet (resp. medial section) of $\overline{\calP}$ be isomorphic to $\calK'$ (resp. $\calN'$). Then the smallest regular cover of $\calP$ is a polytope if and only if $\calK \mix \calK'$ and $\calL \mix \overline{\calL}$ are both polytopes, $\VG{\calK'}{\calK} \cap \VG{\overline{\calL}}{\calL} \leq \VG{\calN'}{\calN}$, and $\VG{\calK}{\calK'} \cap \VG{\overline{\calL}}{\calL} \leq \VG{\calN}{\calN'}$. 
        
        \item Suppose that $I = \{1, \ldots, n-1\}$. Let the base vertex-figure (resp. medial section) of $\overline{\calP}$ be isomorphic to $\calL'$ (resp. $\calN'$). Then the smallest regular cover of $\calP$ is a polytope if and only if $\calK \mix \overline{\calK}$ and $\calL \mix \calL'$ are both polytopes, $\VG{\overline{\calK}}{\calK} \cap \VG{\calL'}{\calL} \leq \VG{\calN'}{\calN}$, and $\VG{\overline{\calK}}{\calK} \cap \VG{\calL}{\calL'} \leq \VG{\calN}{\calN'}$ 
        \end{enumerate}
    \end{corollary}

    \begin{proof}
    First, note that in every case, the smallest regular cover of $\calP$ is $\calP \mix \overline{\calP}$. Whenever $\calP$ is facet-transitive, the base facet of $\overline{\calP}$ is $\overline{\calK}$, and similarly for vertex-figures and medial sections. The first two parts then follow directly from \cref{cor:med-sec-trans-I,cor:2_0n-1}.

    For the third part, observe that in the notation of \cref{cor:2_n-1}, $\calK_{1,1} = \calK_{2,2} = \calK$ and $\calK_{1,2} = \calK_{2,1} = \calK'$. Then the result follows from the observation that $\calK \mix \calK' \cong \calK' \mix \calK$. The fourth part follows similarly.
    \end{proof}

    Note that \cite[Corollary 6.5]{mixing-and-monodromy} already says that if the facets (or vertex-figures) of a two-orbit polytope $\calP$ are all isomorphic to the same regular polytope, then the smallest regular cover of $\calP$ is a polytope. Indeed, in this case $X(\overline{\calK}|\calK)$ is trivial, thus making our condition on variance groups trivially satisfied.

    % \begin{corollary}
    % Let $\calP$ be a %medial-section-transitive 
    % two-orbit $n$-polytope with regular facets and with vertex-figures isomorphic to $\calL$. Then the smallest regular cover of $\calP$ is a polytope if and only if the smallest regular cover of $\calL$ is a polytope.
    % \end{corollary}

    % \begin{proof}
    % If the facets of $\calP$ are isomorphic to the regular polytope $\calK$, then $\overline{\calK} = \calK$. Thus, $\calK \mix \overline{\calK} = \calK$ and $\VG{\calK}{\overline{\calK}}$ is trivial, and the result follows from \cref{gen-src-poly}.
    % \end{proof}

    We note that for some choices of subset $I$, the facets of $\calP \mix \overline{\calP}$ are automatically polytopes, and so there are fewer conditions to check in \cref{gen-src-poly}.
    
    \begin{proposition}
    \label{2I-facets-polytopal}
    Let $\calP$ be an $n$-polytope in class $2_I$ with $n-1 \not \in I$ and $|I| < n-1$. If the facets of $\calP$ are isomorphic to $\calK$, then $\calK \mix \overline{\calK}$ is a polytope.
    \end{proposition}

    \begin{proof}
    Since $|I| < n-1$, $\calP$ is fully transitive. Furthermore, the fact that $n-1 \not \in I$ and that there is some $j < n-1$ such that $j \not \in I$ implies that the facets of $\calK$ are regular (see \cref{prop:n-2-reg}). Then the result follows from \cite[Corollary 6.5]{mixing-and-monodromy}.
    \end{proof}

    \begin{corollary}
    \label{cor:chiral-like-src}
    Suppose that $\calP$ is a medial-section-transitive $n$-polytope in class $2_I$, with $I \cap \{0,n-1\} = \emptyset$. If the facets, vertex-figures, and medial sections of $\calP$ are isomorphic to $\calK, \calL$, and $\calN$, respectively, then the smallest regular cover of $\calP$ is a polytope if and only if  $\VG{\overline{\calK}}{\calK} \cap \VG{\overline{\calL}}{\calL} \leq \VG{\overline{\calN}}{\calN}.$ 
    \end{corollary}

    \begin{proof}
    When $I \cap \{0,n-1\} = \emptyset$, that implies that $|I| < n-1$ and $n-1 \not \in I$, so \cref{2I-facets-polytopal} implies that the facets of $\calP \mix \overline{\calP}$ are polytopal. Similarly, the dual of \cref{2I-facets-polytopal} implies that the vertex-figures of $\calP \mix \overline{\calP}$ are polytopal since $0 \not \in I$. The result then follows from the first part of \cref{gen-src-poly}.
    \end{proof}

    Note that in particular, chiral polytopes satisfy the conditions of \cref{cor:chiral-like-src}, and so we now obtain \cite[Theorem 3.1]{chiral-covers} as a corollary.

\subsection{Mixes of $\calT$-admissible maniplexes with non-$\calT$-admissible maniplexes}

    How should we proceed if we want to mix two maniplexes that have no $\calT$ such that they are both $\calT$-admissible? Let us focus for a moment on a special case: suppose that $\calM_1$ is in class $2_I$ and that $\calM_2$ is regular and $I$-non-orientable. Then since $\calM_1$ covers $\textbf{2}_I$, it follows that
    \[ \calM_1 \mix \calM_2 \cong (\calM_1 \mix \textbf{2}_I) \mix \calM_2 \cong \calM_1 \mix (\textbf{2}_I \mix \calM_2). \]
    If $\calM_1$ is a polytope, then we can apply \cref{thm:all-i-admissible-mix} by replacing $\calM_2$ with its $I$-double $\textbf{2}_I \mix \calM_2$. On the other hand, if $\calM_2$ is a polytope and $\calM_1$ is not, then we can only apply \cref{thm:all-i-admissible-mix} if $\textbf{2}_I \mix \calM_2$ is a polytope; see \cref{i-double-poly}. 

    The above argument can be generalized to the case where $\calM_1$ is $\calT$-admissible and $\calM_2$ is regular but not $\calT$-admissible; in that case,
    \[ \calM_1 \mix \calM_2 \cong \calM_1 \mix (\calT \mix \calM_2),\]
    which is now the mix of two $\calT$-admissible premaniplexes. In principle, this could be generalized further to the case where $\calM_2$ is `$\calT$-rotary' but not `$\calT$-orientable'. 

    We conclude by presenting several open problems that would be natural candidates for further investigation.

    \begin{problem}
    Generalize \cref{i-double-poly} to characterize when a $\calT$-admissible cover of a regular polytope is a polytope.
    \end{problem}

    \begin{problem}
    Describe conditions under which the mix of two polytopes in class $2_I$ is still in class $2_I$ (rather than being regular). 
    \end{problem}

    \begin{problem}
    Determine general criteria for when the smallest regular cover of a polytope is a polytope.
    \end{problem}

\bibliographystyle{amsplain}
\bibliography{2-orbit}

\providecommand{\bysame}{\leavevmode\hbox to3em{\hrulefill}\thinspace}
\providecommand{\MR}{\relax\ifhmode\unskip\space\fi MR }
\providecommand{\MRhref}[2]{%
  \href{http://www.ams.org/mathscinet-getitem?mr=#1}{#2}
}
\providecommand{\href}[2]{#2}
\begin{thebibliography}{10}

\bibitem{chiral-mix}
Antonio Breda~D'Azevedo, Gareth Jones, and Egon Schulte, \emph{Constructions of
  chiral polytopes of small rank}, Canad. J. Math. \textbf{63} (2011), no.~6,
  1254--1283.

\bibitem{var-gps}
Gabe Cunningham, \emph{Variance groups and the structure of mixed polytopes},
  Rigidity and symmetry, Fields Inst. Commun., vol.~70, Springer, New York,
  2014, pp.~97--116. \MR{3329271}

\bibitem{chiral-covers}
\bysame, \emph{Chiral polytopes whose smallest regular cover is a polytope},
  Journal of Combinatorial Theory, Series A \textbf{204} (2024), 105839.

\bibitem{stg}
Gabe Cunningham, Mar{\'i}a Del R{\'i}o-Francos, Isabel Hubard, and Micael
  Toledo, \emph{Symmetry type graphs of polytopes and maniplexes}, Ann. Comb.
  \textbf{19} (2015), no.~2, 243--268. \MR{3347382}

\bibitem{k-orbit}
Gabe Cunningham and Daniel Pellicer, \emph{Open problems on {$k$}-orbit
  polytopes}, Discrete Math. \textbf{341} (2018), no.~6, 1645--1661.
  \MR{3784786}

\bibitem{poly-mani}
Jorge Garza-Vargas and Isabel Hubard, \emph{Polytopality of maniplexes},
  Discrete Math. \textbf{341} (2018), no.~7, 2068--2079. \MR{3802160}

\bibitem{isa2orbit}
Isabel Hubard, \emph{Two-orbit polyhedra from groups}, European Journal of
  Combinatorics \textbf{31} (2010), no.~3, 943--960.

\bibitem{voltage-ops}
Isabel Hubard, El{\'\i}as Moch{\'a}n, and Antonio Montero, \emph{Voltage
  operations on maniplexes, polytopes and maps}, Combinatorica \textbf{43}
  (2023), no.~2, 385--420.

\bibitem{sparse}
Isabel Hubard and Micael Toledo, \emph{Sparse groups need not be semisparse},
  Aequationes mathematicae (2024), 1--24.

\bibitem{flag-bicolorings}
Hiroki Koike, Daniel Pellicer, Miguel Raggi, and Steve Wilson, \emph{Flag
  bicolorings, pseudo-orientations, and double covers of maps}, The Electronic
  Journal of Combinatorics \textbf{24} (2017), no.~1, P1--3.

\bibitem{faithful-thin}
Dimitri Leemans and Micael Toledo, \emph{Faithful and thin non-polytopal
  maniplexes}, Ars Mathematica Contemporanea \textbf{25} (2025), no.~1.

\bibitem{arp}
Peter McMullen and Egon Schulte, \emph{Abstract regular polytopes},
  Encyclopedia of Mathematics and its Applications, vol.~92, Cambridge
  University Press, Cambridge, 2002. \MR{1965665 (2004a:52020)}

\bibitem{mixing-and-monodromy}
Barry Monson, Daniel Pellicer, and Gordon Williams, \emph{Mixing and monodromy
  of abstract polytopes}, Transactions of the American Mathematical Society
  \textbf{366} (2014), no.~5, 2651--2681.

\bibitem{map-operations}
Alen Orbani{\'c}, Daniel Pellicer, and Asia~Ivi{\'c} Weiss, \emph{Map
  operations and k-orbit maps}, Journal of Combinatorial Theory, Series A
  \textbf{117} (2010), no.~4, 411--429.

\bibitem{parallel-product}
Stephen~E. Wilson, \emph{Parallel products in groups and maps}, J. Algebra
  \textbf{167} (1994), no.~3, 539--546. \MR{1287058 (95d:20067)}

\bibitem{maniplexes}
Steve Wilson, \emph{Maniplexes: {P}art 1: maps, polytopes, symmetry and
  operators}, Symmetry \textbf{4} (2012), no.~2, 265--275. \MR{2949129}

\end{thebibliography}

\end{document}